\documentclass[12pt,headings=big]{scrartcl}
 \usepackage{cmap}
 \usepackage{amsmath,amsxtra,amscd,amssymb,latexsym,stmaryrd,amsthm,lipsum}
 \usepackage{mathrsfs, calc, xcolor, array, graphicx, subcaption}

\allowdisplaybreaks

\usepackage[colorlinks=true,linkcolor=red!70!black,citecolor=green!70!black,
    urlcolor=magenta!70!black,backref]{hyperref}

  \usepackage{lmodern} 
  \usepackage[utf8]{inputenc}
  \usepackage[T1]{fontenc}

\theoremstyle{plain}
\newtheorem{theo}{Theorem}[section]
\newtheorem*{theo*}{Theorem}
\newtheorem{coro}[theo]{Corollary}
\newtheorem{prop}[theo]{Proposition}
\newtheorem{lemm}[theo]{Lemma}
\newtheorem{theomain}{Theorem}
\newtheorem{coromain}[theo]{Corollary}

\newtheorem{hypo}[theo]{Standing assumption}
\newtheorem*{nota}{Notation}

\theoremstyle{definition}

\newtheorem{rema}[theo]{Remark}

\newcommand*{\dd}%
  {\relax\ifnum\lastnodetype>0\mskip\medmuskip\fi\mathrm{d}}
\newcommand{\fspace}[1]{\mathscr{#1}}
\newcommand{\fX}{\fspace{X}}

\newcommand{\PM}{\operatorname{\mathcal{P}}}
\newcommand{\op}[1]{\mathrm{#1}}
\newcommand{\one}{\boldsymbol{1}}

\newcommand{\mchain}{\mathsf}

\newcommand{\pr}{\operatorname{\mathbb{P}}}
\newcommand{\esp}{\operatorname{\mathbb{E}}}
\newcommand{\Lip}{\operatorname{Lip}}
\newcommand{\BV}{\operatorname{BV}}

\newcommand{\Hol}{\operatorname{Hol}}

\newlength{\hypobox}
\newlength{\gapbox}
\newcounter{hypop}
\renewcommand{\thehypop}{\textup{(H\arabic{hypop}')}}

\newcounter{gap}
\renewcommand{\thegap}{\textup{(SG\arabic{gap})}}

\title{Effective limit theorems for \\ Markov chains with a spectral gap}

\author{Beno\^{\i}t R. Kloeckner \thanks{Universit\'e Paris-Est, Laboratoire d'Analyse et de Mat\'ematiques Appliqu\'ees (UMR 8050), UPEM, UPEC, CNRS, F-94010, Cr\'eteil, France}}

\begin{document}

\maketitle

\begin{abstract}
Applying quantitative perturbation theory for linear operators, we prove non-asymptotic bounds for Markov chains whose transition kernel has a spectral gap in an arbitrary Banach algebra of functions $\fspace{X}$. The main results are concentration inequalities and Berry-Esseen bounds, obtained assuming neither reversibility nor ``warm start'' hypothesis: the law of the first term of the chain can be arbitrary. The spectral gap hypothesis is basically a uniform $\fspace{X}$-ergodicity hypothesis, and when $\fspace{X}$ consist in regular functions this is weaker than uniform ergodicity.
We show on a few examples how the flexibility in the choice of function space can be used. The constants are completely explicit and reasonable enough to make the results usable in practice, notably in MCMC methods.
\end{abstract}

\section{Introduction}

\paragraph{General framework}
Let $(X_k)_{k\ge 0}$ be a Markov chain taking value in a general state space $\Omega$, and let $\varphi:\Omega \to \mathbb{R}$ be a function (the ``observable''). Under rather general assumptions, there is a unique stationary measure $\mu_0$ and it can be proved that almost surely\footnote{Here and in the sequel, we write indifferently $\mu(f)$ or $\int f\dd\mu$ for the integral of $f$ with respect to the measure $\mu$.}
\begin{equation}
\frac1n \sum_{k=1}^n \varphi(X_k) \to \mu_0(\varphi).
\label{eq:LLN}
\end{equation}
Then a natural question is to ask at what speed this convergence occurs. In many cases, one can prove a Central Limit Theorem, showing that the convergence has the order $1/\sqrt{n}$. But this is again an asymptotic result, and one is led to ask for non-asymptotic bounds, both for the Law of Large Numbers \eqref{eq:LLN} (``concentration inequalities'') and for the CLT (``Berry-Esseen bounds'').

\paragraph{A word on effectivity}
In this paper, the emphasis will be on \emph{effective} bounds, i.e. given an explicit sample size $n$, one should be able to deduce from the bound that the quantity being considered lies in some explicit interval around its limit with at least some explicit probability. In other words, the result should be \emph{non-aymptotic} and all constants should be made explicit. The motivations for this are at least twofold. 

First, in practical applications of the Markov chain Monte-Carlo (MCMC) method, where one uses \eqref{eq:LLN} to estimate the integral $\mu_0(\varphi)$, effective results are needed to obtain proven convergence of a given precision. MCMC methods are important when the measure of interest is either unknown, or difficult to sample independently (e.g. uniform in a convex set in large dimension), but happens to be the stationary measure for an easily simulated Markov chain. The Metropolis-Hastings algorithm for example makes it possible to deal with an absolutely continuous measure whose density is only known up to the normalization constant.

A second, more theoretical motivation is that the constants appearing in limit theorem depend on a number of parameters (e.g. the mixing speed of the Markov chain, the law of $X_0$, etc.). When the constants are not made explicit, one may not be able to deduce from the result how the convergence speed changes when some parameter approaches the limit of the domain where the result is valid (e.g. when the spectral gap tends to $0$).

There are many works proving concentration inequalities and (to a lesser extent) Berry-Esseen bounds for Markov chains, under a variety of assumptions, and we will only mention a small number of them. To explain the purpose of this article, let us discuss briefly three directions. 

\paragraph{Previous works (1): total variation convergence}
The first direction is mainly motivated by MCMC; we refer to \cite{Roberts-Rosenthal} for a detailed introduction to the topic. 

The Markov chains being considered are usually ergodic (either \emph{uniformly}, which corresponds to a spectral gap on $L^\infty$, or \emph{geometrically}); one measures difference between probability measure using the \emph{total variation distance}, and the limit theorems are typically obtained for $L^\infty$ observables $\varphi$ (the emphasis here is not on the boundedness, but on the lack of regularity assumption). Effective concentration inequalities have been obtained in this setting, for example in \cite{Glynn2002} and \cite{Kontoyiannis-LMM} which we shall discuss below. Watanabe and Hayashi \cite{watanabe2017finite} have given bounds for tail probability and applied this to hypothesis testing, but their method is restricted to finite-state spaces. Berry-Esseen bounds have been proved in \cite{Bolthausen}, but effective results are less common. 

\paragraph{Previous works (2): the spectral method}
The second direction grew from the ``Nagaev method'' \cite{Nagaev1,Nagaev2}, a functional approach where perturbative spectral theory enables one to adapt the classical Fourier proofs of limit theorems, from independent identically distributed random variable to suitable Markov chains. This approach is described in \cite{HH} in a quite general setting, and is especially popular in dynamical systems (the statistical properties of certain dynamical systems can be studied more easily by reversing time, and considering a Markov chain jumping randomly along \emph{backward} orbits).  

There, the Markov chain being considered are often \emph{not} ergodic in the total variation sense, but instead their transition kernel has a spectral gap in a space $\fspace{X}$ made of regular (e.g. Lipschitz or H\"older) functions; one sometimes say such a Markov chain is $\fspace{X}$-ergodic. The limit theorems are then restricted to observables $\varphi\in\fspace{X}$, and the speed of convergence is driven by the regularity of $\varphi$ as much as by its magnitude. Due to the use of perturbation theory of operator, in most cases this method has not yielded \emph{effective} results. 

Note that the spectral method \emph{can} be applied without regularity assumptions, taking e.g. $\fspace{X}=L^2(\mu_0)$ or $\fspace{X}=L^\infty(\Omega)$ (or variants, see \cite{Kontoyiannis-M}), thus the present direction intersects the previous one.

There are a few exceptions to the aforementioned lack of effectiveness. When $\fspace{X}$ is a Hilbert space, by symetrization of the transition kernel one can use well-known effective perturbation results. In this way, Lezaud obtains effective concentration inequalities and Berry-Esseen bounds \cite{Lezaud,Lezaud2001}, see also \cite{Paulin2015}. Both work in $L^2(\mu_0)$, restricting accordingly the Markov chains that can be considered. Second Dubois \cite{Dubois} gave what seems to be the first effective Berry-Esseen inequality in a dynamical context, and we shall compare the present Berry-Esseen inequality with his. Last, Liverani \cite{liverani2001rigorous} made very explicit the perturbation result obtained with Keller \cite{Keller-Liverani} for operators in ``strong-to-weak'' norms, which might be usable to obtain concentration results.

\paragraph{Previous works (3): Lipschitz observables}
The third direction is quite recent: Joulin and Ollivier \cite{JO} used ideas from optimal transportation to prove very efficiently effective concentration results under a \emph{positive curvature} hypothesis; this corresponds to strict contraction on the space $\fspace{X}=\Lip$ of Lipschitz functions.  Paulin \cite{Paulin2016} extended this method to the slightly more general case of a spectral gap (on the same space). In a similar context but with different methods, Dedeker and Fan \cite{dedeker2015deviation} proved concentration near the expectation for non-linear, separately Lipschitz functionals.

This method is very appealing, but is restricted to a single, pretty restrictive function space constraining both the Markov chains and the observables that can be considered; we will see in examples below that being able to change the function space can be useful to get good constants even when \cite{JO} can be applied. Moreover, this method seems unable to provide higher-order limit theorem such as the CLT or Berry-Esseen bounds.

\paragraph{Contributions of this work}
The goal of this article is to combine recent \emph{effective} perturbation results \cite{K:perturbation} with the Nagaev method to obtain effective concentration inequalities and Berry-Esseen bounds for a wealth of Markov chains. Our main hypothesis will basically be a spectral gap on some function space $\fspace{X}$, with the restriction that we need $\fspace{X}$ to be a Banach \emph{algebra} (this will in particular restrict us to bounded observables).
We obtain three main results: 
\begin{itemize}
\item a general concentration inequality (Theorem \ref{theo:main-conc}),
\item a variant which, under a bound on the \emph{dynamical variance} of $(\varphi(X_k))_k$, gives an optimal rate for small enough deviations (Theorem \ref{theo:main-second}),
\item a general Berry-Esseen bound (Theorem \ref{theo:main-BE}).
\end{itemize}

Let us give a few examples where our results apply:
\begin{itemize}
\item taking $\fspace{X}=L^\infty(\Omega)$, our assumptions essentially reduce to uniform ergodicity of the Markov chain and boundedness of the observable,
\item taking $\fspace{X}=\Lip(\Omega)$, our assumptions essentially reduce to positively curved Markov chains (in the sense of Ollivier) and bounded Lipschitz observables. This for example applies to contracting Iterated Function Systems and backward random walks of expanding maps. We shall see (Section \ref{sec:cube}) that in the toy case of the discrete hypercube and observables with small Lipschitz constant, Theorem \ref{theo:main-conc} is less powerful than \cite{JO} but that for larger Lipschitz constants, Theorem \ref{theo:main-second} can improve on \cite{JO},
\item when $\Omega$ is a graph, we propose a functional space of functions with small ``local total variations'' that yields improvement over \cite{JO} in some cases (Section \ref{sec:cube}),
\item taking $\fspace{X}=\BV(I)$ where $I$ is an interval, our results apply to a natural Markov chains related to \emph{Bernoulli convolutions}, allowing observables of bounded variation such as characteristic functions of intervals (Section \ref{sec:Bernoulli}),
\item more generally, when $\Omega$ is a domain of $\mathbb{R}^d$ some natural Markov chains are $\BV(\Omega)$-ergodic and our results apply to functions of bounded variation, e.g. characteristic functions of sets of finite perimeter -- but we will not consider this case here, since it needs a somewhat sophisticated setup,
\item Another direction we do not explore here is to take $\fspace{X}=\Hol_\alpha(\Omega)$, the space of $\alpha$-H\"older functions, or in case $\Omega=I$ is an interval, $\fspace{X}=\BV_p(I)$, the space of $p$-bounded variation functions. These enable one to consider more general functions than $\Lip(\Omega)$ or respectively $\BV(\Omega)$; even for Lipschitz of BV functions, using these spaces can be useful because they tend to give regular observables a much lower norm.
\end{itemize}

To my knowledge, no effective result was known in the setting of bounded variation functions (and while the usual spectral method could have been used in this case, I do not know of previous asymptotic results either) and the effective Berry-Esseen bound seems new in most of the above cases.

\paragraph{Structure of the article} In Section \ref{sec:results} we state notation and the main results. Section \ref{sec:examples} explains briefly the aforementioned examples and compares our results with previous ones; detailed proofs are available in a companion note \cite{K:examples}. In Section \ref{sec:perturbation} we recall how perturbation theory can be used to prove limit theorems, and state the perturbation results we need to carry out this method in a effective manner.
In Section \ref{sec:main} we prove the core estimates to be used thereafter, while Section \ref{sec:concentration} carries out the proof of the concentration inequalities. Section \ref{sec:BE} is devoted to the proof of the Berry-Esseen inequality.

\section{Assumptions and main results}
\label{sec:results}

Let $\Omega$ be a Polish metric space endowed with its Borel $\sigma$-algebra and denote by $\PM(\Omega)$ the set of probability measures on $\Omega$.
We consider a transition kernel $\mchain{M}=(m_x)_{x\in\Omega}$ on $\Omega$, i.e. $m_x\in \PM(\Omega)$ for each $x\in\Omega$, and a Markov chain $(X_k)_{k\ge 0}$ following the kernel $\mchain{M}$, i.e. $\pr(X_{k+1}\mid X_k=x)=m_x$. We will only consider cases where there exist a unique stationary measure (see Remark \ref{rema:stationary} below), but we do not ask the Markov chain to be stationary: the law of $X_0$ is arbitrary (``cold start''). In some cases of interest, the law of each $X_k$ will even be singular with respect to the stationary measure.

\begin{nota}
In the following, $\mu_0$ will always denote the stationary measure of $\mchain{M}$, and $\mu$ shall denote the law of $X_0$ (which is arbitrary).
\end{nota}

We shall study the behavior of $(X_k)_{k\ge0}$ by comparing the empirical mean to the stationary mean:
\[ \hat\mu_n(\varphi) := \frac1n \sum_{k=1}^n \varphi(X_k) \quad\text{vs.}\quad  \mu_0(\varphi)\]
for an arbitrary ``observable'' $\varphi\in \fspace{X}$, where $\fspace{X}$ is a space of functions $\Omega\to \mathbb{R}$ (or $\Omega\to \mathbb{C}$). Our method seems not (directly) suitable to consider more general, non-linear functionals $\Phi(X_1,\dots,X_n)$: we decompose $\hat\mu_n(\varphi)$ to make a power of a perturbed transfer operator appear (see Section \ref{sec:perturbation}).

\subsection{Assumptions}

\begin{hypo}\label{hypo:X}
In all the paper, we assume $\fspace{X}$ satisfies the following:
\begin{enumerate}
\item its norm  $\lVert\cdot \rVert$ dominates the uniform norm: $\lVert\cdot \rVert\ge \lVert\cdot \rVert_\infty$,
\item $\fspace{X}$ is a Banach algebra, i.e. for all $f,g\in \fspace{X}$ we have $\lVert fg\rVert \le \lVert f\rVert \lVert g \rVert$,
\item $\fspace{X}$ contains the constant functions and
$\lVert\one\rVert= 1$ (where $\one$ denotes the constant function with value $1$).
\end{enumerate}
\end{hypo}
The first hypothesis ensures integrability with respect to arbitrary probability measure, which is important for cold-start Markov chains; it also implies that every probability measure can be seen as a continuous linear form acting on $\fspace{X}$.
The second hypothesis will prove very important in our method where products abound (and can be replaced by the more lenient $\lVert fg\rVert \le C\lVert f\rVert \lVert g \rVert$ up to multiplying the norm by a constant), and the hypothesis on $\lVert\one\rVert$ is a mere matter of convenience and could be removed at the cost of more complicated formulas.

\begin{rema}
This setting may seem restrictive at first: the Banach algebra hypothesis notably excludes $L^p$ spaces, while classically one only makes moment assumptions on the observable. This is quite unavoidable given that we will work with more than one equivalence class of measures, and we want to allow cold start at a given position ($X_0\sim \delta_{x_0}$). The measures $m_x$ may be singular with respect to the stationary measure $\mu_0$, and as a matter of fact in the dynamical applications $m_x$ will be purely atomic while $\mu_0$ will often be atomless. It may thus happen that for $\varphi$ an $L^p(\mu_0)$ observable, $\varphi(X_j)$ is undefined with positive probability, or is extremely large even if $\varphi$ has small moments with respect to $\mu_0$.
\end{rema}

To the transition kernel $\mchain{M}$ is associated an averaging operator acting on $\fspace{X}$:
\[\op{L}_0 f(x) = \int_\Omega f(y) \dd m_{x}(y).\]
Since each $m_x$ is a probability measure, $\op{L}_0$ has $1$ as eigenvalue, with eigenfunction $\one$. 

\begin{hypo}\label{hypo:L}
In all the article we assume $\mchain{M}$ satisfies the following:
\begin{enumerate}
\item $\op{L}_0$ acts as a bounded operator from $\fspace{X}$ to itself, and its operator norm $\lVert \op{L}_0\rVert$ is equal to $1$.
\item $\op{L}_0$ is contracting with gap $\delta_0>0$, i.e. there is a closed hyperplane $G_0 \subset \fspace{X}$ such that
\[ \lVert \op{L}_0 f \rVert \le (1-\delta_0) \lVert f\rVert \qquad \forall f\in G_0.\]
\end{enumerate}
\end{hypo}
The first hypothesis could be relaxed, considering operators of arbitrary norm, at the cost of more complicated formulas.

\begin{rema}\label{rema:gap}
The second hypothesis is the main one, and implies in particular that $1$ is a simple isolated eigenvalue. It is a slightly stronger assumption than a spectral gap, which can be written as
\[ \lVert \op{L}_0^n f \rVert \le C(1-\delta_0)^n \lVert f\rVert \qquad \forall f\in G_0\]
for all $n\in\mathbb{N}$ and some $C\ge 1$ (what we call here a contraction with gap $\delta_0$ can thus also be called a spectral gap of size $\delta_0$ with constant $1$). When $\op{L}_0$ only has a spectral gap, all our results still apply to the Markov chains $Y_m = X_{n_0+mk}$ where $n_0$ is arbitrary and $k$ is such that $C(1-\delta_0)^k < 1$. This trick can be also used when $C=1$, in cases where the gap is small; in numerical computations, this can be especially useful when the simulation of the random walk is much cheaper than the evaluation of the observable.
\end{rema}

\begin{rema}\label{rema:stationary}
The contraction hypothesis (or a mere spectral gap) ensures that up to scalar factors there is a unique continuous linear form $\phi_0$ acting on $\fspace{X}$ such that  $\phi_0 \circ \op{L}_0=\phi_0$; since any stationary measure of $\mchain{M}$ satisfy this, all stationary measures coincide on $\fspace{X}$. They might not be unique (e.g. if $\fspace{X}$ contains only constants), but since we consider the $\varphi(X_k)$ with $\varphi\in\fspace{X}$, this will not matter. We will thus denote an arbitrary stationary measure by $\mu_0$, and identify it with $\phi_0$ (observe that $G_0$ is then equal to $\ker\mu_0$). In most cases, $\fspace{X}$ will be dense in the space of continuous function endowed with the uniform norm, ensuring that two measures coinciding on $\fspace{X}$ are equal, and then the contraction hypothesis ensures the uniqueness of the stationary measure.
\end{rema}

\begin{rema}\label{rema:assumptions}
There are numerous examples where assumptions \ref{hypo:X} and \ref{hypo:L} are satisfied; we will present a few of them in Section \ref{sec:examples}. Typically, $\fspace{X}$ has a norm of the form $\lVert\cdot\rVert = \lVert \cdot\rVert_\infty+V(\cdot)$ where $V$ is a seminorm measuring the regularity in some sense (e.g. Lipschitz constant, $\alpha$-H\"older constant, total variation, total $p$-variation...) and satisfying $V(fg)\le \lVert f\rVert_\infty V(g)+ V(f)\lVert g \rVert_\infty$. This inequality ensures that $\fspace{X}$ is a Banach Algebra, and $\lVert \one\rVert=1$ holds as soon as $V(\one)=0$. Since averaging operators necessarily satisfy $\lVert \op{L}_0 f\rVert_\infty\le \lVert f\rVert_\infty$, it is sufficient that $\op{L}$ contracts $V$ (i.e. $V(\op{L}_0 f)\le \theta V(f)$ for some $\theta\in(0,1)$ and all $f\in\fspace{X}$) to ensure that $\lVert\op{L}_0\rVert =1$. It can be proved that in many cases, the contraction of $V$ also implies the contraction of $\lVert\cdot\rVert$ in the sense of assumption \ref{hypo:L} (see Lemma 2.3 of \cite{K:HT}, and a more general version in \cite{K:examples}). In fact, all examples considered here are of this kind, but it seemed better to state our main results in terms of the hypotheses we use directly in the proof. This is done at the expense of some sharpness: indeed we could in some cases improve our constants by estimating with more precision $\lVert\pi_0\rVert$ below (see Lemma 2.4 of \cite{K:HT}).
\end{rema}

\subsection{Concentration inequalities}

Our first result is a concentration inequality, featuring a dichotomy between a Gaussian regime and an exponential regime (note that we consider concentration near $\mu_0(\varphi)$: in many cases there is a purely Gaussian concentration near $\esp[\hat\mu_n(\varphi)]$, and the exponential regime appears due to the bias $\mu_0(\varphi) - \esp[\hat\mu_n(\varphi)]$).

\begin{theomain}\label{theo:main-conc}
For all $n\ge 1+\frac{\log 100}{-\log(1-\delta_0/13)}$ it holds:
\[ \pr_\mu\Big[\lvert\hat\mu_n(\varphi)-\mu_0(\varphi)\rvert\ge a\Big]    \le \begin{cases}\displaystyle
2.488 \exp\Big(-n  \frac{\delta_0}{13.44\delta_0+8.324}  \frac{a^2}{\lVert\varphi\rVert^2}\Big)
 & \\
 \hfill \displaystyle \mbox{if }\frac{a}{\lVert \varphi\rVert} \le \frac{\delta_0}{3} \\[6\jot]
 \displaystyle
2.624 \exp\Big( -n\frac{0.98 \delta_0^2}{12+13\delta_0} \Big(\frac{a}{\lVert\varphi\rVert}-0.254\delta_0\Big) \Big) &\\
 \hfill \mbox{otherwise.}
\end{cases}\]
\end{theomain}

See Section \ref{sec:examples} and \cite{K:examples} for a few sample cases where this result applies and comparisons with previous results. Let us stress right away that the main strength of the present result is its broadness: we need no warm-start hypothesis, no reversibility, and we can apply it in many functional spaces. In particular, this makes our results broader than those of \cite{Lezaud,Lezaud2001} which assume ergodicity. Lezaud also gets a front constant proportional to the $L^2(\mu_0)$-norm of the density of the distribution of $X_0$ with respect to the stationary distribution, which would be infinite in many of our cases of applicability;  even in the case of a finite state space he then gets a large front constant when $X_0\sim \delta_x$. The approach of Joulin and Ollivier enabled them to get rid of this constant in some test cases, and we compare our results to theirs in Section \ref{sec:cube}.

The spectral method gives us access to higher-order estimates, enabling us to improve the Gaussian regime bound as soon as we have a good control over the ``dynamical variance'' (also called ``asymptotic variance'') $\sigma^2(\varphi)$, which is the variance appearing in the CLT for $(\varphi(X_k))_{k\ge 0}$; setting $\bar\varphi = \varphi-\mu_0(\varphi)$, the dynamical variance is defined by:
\[\sigma^2(\varphi)=\mu_0(\varphi^2)-(\mu_0\varphi)^2+2\sum_{k\ge 1} \mu_0(\varphi \op{L}_0^k \bar\varphi).\]

\begin{theomain}\label{theo:main-second}
Whenever $n\ge 1+\frac{\log 100}{-\log(1-\delta_0/13)}$, $U \ge \sigma^2(\varphi)$ and 
$a \le \frac{U}{\lVert\varphi\rVert} \log\Big(1+\frac{\delta_0^2}{12+13\delta_0}\Big)$,
\[\pr_\mu\big[\lvert \hat\mu_n(\varphi)-\mu_0(\varphi)\rvert\ge a\big]
  \le 2.637 \exp\left(-n\cdot\Big(\frac{a^2}{2U} - 10(1+\delta_0^{-1})^2 \frac{\lVert\varphi\rVert^3a^3}{U^3}\Big)\right).\]
Given an upper bound $S\ge \sigma^2(\varphi)$, the right-hand side of the above inequality is minimized for $U\in \big[max(S,a\cdot\lVert\varphi\rVert/\log(1+\delta_0^2/(12+13\delta_0))),\infty\big)$ at
\[U_\mathrm{min} := \max\Big( S, \sqrt{a}\cdot\sqrt{60}(1+\delta_0^{-1})\lVert\varphi\rVert^{\frac32}, a\cdot \frac{\lVert\varphi\rVert}{\log(1+\delta_0^2/(12+13\delta_0))} \Big).\]
By this substitution, the reader can easily get a bound only in terms of $a$ and $S$.
\end{theomain}

For small enough $a$, one takes $U=S$ and the positive term in the exponential is negligible; the leading term $-na^2/(2S)$ is then exactly the best we can expect given the bound $S$: since $(\varphi(X_k))_k$ satisfies a Central Limit Theorem with variance $\sigma^2(\varphi)$, any better value would necessarily imply a better bound on $\sigma^2(\varphi)$.

Paulin \cite{Paulin2015} (Theorem 3.3) obtained a similar result for stationary, reversible Markov Chains with a spectral gap in $L^2$; the advantage of our result is to dispense from stationarity, reversibility, and to apply to various functional spaces.

Section \ref{sec:cube} contains an example where Theorem \ref{theo:main-second} improves crucially on Theorem \ref{theo:main-conc}. However bounding the dynamical variance can be difficult in general. In practice, one could use other tools to estimate it and then apply Theorem \ref{theo:main-second}.

\subsection{A Berry-Esseen bound}

Our third main result, proven in section \ref{sec:BE}, quantifies the speed of convergence in the Central Limit Theorem.
\begin{theomain}\label{theo:main-BE}
Assume $\sigma^2(\varphi)>0$ and let 
$\tilde\varphi:=\frac{\varphi-\mu_0(\varphi)}{\sigma(\varphi)}$
be the reduced centered version of $\varphi$, and denote by
$G, F_n$ the distribution functions of the reduced centered normal law and of $\frac{1}{\sqrt{n}}(\tilde\varphi(X_1)+\dots+\tilde\varphi(X_n))$, respectively.
For all $n\ge 1$ it holds
\[\lVert F_n-G\rVert_\infty \le \frac{(148+285\delta_0^{-1}+123 \delta_0^{-2}) \max\{\lVert\tilde\varphi\rVert,\lVert\tilde\varphi\rVert^3\}}{\sqrt{n}} .\]
\end{theomain}

The absence of a lower bound for $n$ simply comes from the fact that for small $n$, the right-hand side is greater than $1$ (see Lemma \ref{lemm:away0}) and the inequality is thus vacuously true.

\begin{rema}
Note that $\sigma^2(\varphi)$ is always non-negative, as it can be rewritten as
\[\lim_{n\to\infty} \frac1n \mathrm{Var}_{\mu_0}\big(\sum_{k=1}^n \varphi(X_k)\big)\]
(where the $\mu_0$ subscript means that the assumption $X_0\sim \mu_0$ is made). However, $\sigma^2(\varphi)$ can vanish even when $\varphi$ is not constant modulo $\mu_0$, as in the case of a dynamical system when $m_x$ is supported on $T^{-1}(x)$ for some map $T:\Omega\to\Omega$, and $\varphi$ is a coboundary: $\varphi=g-g\circ T$ for some $g$. One can for example see details \cite{GKLM}, where $\sigma^2$ is interpreted as a semi-norm. Whenever $\sigma^2(\varphi)=0$, one can use the present method to obtain stronger  non-asymptotic concentration inequalities, giving small probability to deviations $a$ such that $a/\lVert\varphi\rVert \gg 1/n^{2/3}$ instead of $a/\lVert\varphi\rVert \gg 1/\sqrt{n}$.
\end{rema}

There are numerous works on Berry-Esseen bounds. In the case of independent identically distributed random variables, the optimal constant is not yet known (the best known constant is, to my knowledge, given by Tyurin \cite{Tyurin}). Berry-Esseen bounds for Markov chains go back to \cite{Bolthausen}, but I know only of two previous \emph{effective} results, by Dubois \cite{Dubois} and by Lezaud \cite{Lezaud2001}.

The scope of Dubois' result is quite narrower than ours, as it is only written for uniformly expanding maps of the interval and Lipschitz observables (though the method is expected to have wider application), and our numerical constant is much better: while the dependences on the parameters of the system are stated differently and thus somewhat difficult to compare, Dubois has a front constant of 11460 which is quite large for practical applications (the order of convergence being $1/\sqrt{n}$, this constant has a squared effect on the number of iterations needed to achieve a given precision).

The scope of Lezaud's Berry-Esseen bound is also restricted, to ergodic reversible Markov chains. Moreover he gets a front constant proportional to the $L^2(\mu_0)$-norm of the density of the distribution of $X_0$ with respect to the stationary distribution; in comparison, our result is insensitive to the distribution of $X_0$.

\paragraph{Application to dynamical systems}
As is well-known, limit theorems for Markov chain also apply in a dynamical setting (see e.g. \cite{gouezel2015limit}).
 Given a $k$-to-one map $T:\Omega\to \Omega$, one defines the transfer operator of a potential $A\in\fspace{X}$ by
\[\op{L}_{T,A} f(x) = \sum_{y\in T^{-1}(x)} e^{A(y)} f(y).\]
One says that $A$ is normalized when $\op{L}_{T,A}\one=\one$. This condition exactly means that $m_x=\sum_{y\in T^{-1}(x)} e^{A(y)} \delta_y$ is a probability measure for all $x$, making $\op{L}_{T,A}$ the averaging operator of a transition kernel. We could consider more general maps $T$, considering a transition kernel that is supported on its inverse branches.

If the transfer operator has a spectral gap, then the stationary measure $\mu_0$ is unique, and readily seen to be $T$-invariant. We shall denote it by $\mu_A$ to stress the dependence on the potential. The corresponding stationary Markov chain $(Y_k)_{k\in\mathbb{N}}$ satisfies all results presented above; but for each $n$, the time-reversed process defined by $X_k=Y_{n-k}$ (where $0\le k \le n$) satisfies $X_{k+1}=T(X_k)$: all the randomness lies in $X_0=Y_n$. Having taken $Y_n$ stationary makes the law of $Y_n$, i.e. $X_0$, independent of the choice of $n$. It follows:
\begin{coromain}
For all normalized $A\in \fspace{X}$ such that $\op{L}_{T,A}$ is contracting with gap $\delta_0$, for all $\varphi\in\fspace{X}$, Theorems \ref{theo:main-conc}, \ref{theo:main-second} and \ref{theo:main-BE} hold for the random process
$(X_k)_{k\in\mathbb{N}}$ defined by $X_0\sim \mu_A$ and $X_{k+1}=T(X_k)$.
\end{coromain}
In this context, spectral gap was proved in many cases under the impetus of Ruelle, see e.g. the books \cite{Baladi,Ruelle}, the recent works \cite{Bruin-Todd,CV,Cyr-Sarig09}, and references therein. Chazottes and Gou\"ezel \cite{chazottes2012optimal} proved concentrations inequalities for non-uniformly hyperbolic dynamical systems, but with a non-explicit constant.

Let me finally mention \cite{K:HT} (based on the same effective perturbation theory as the present paper) and \cite{K:WeaklyExp}.

\section{Examples}\label{sec:examples}

In this Section we briefly present some basic examples where our results apply; detailed proofs of the claims can be found in the note \cite{K:examples}.

\subsection{Chains with Doeblin's minorization}\label{sec:Doeblin}

The simplest example of a Banach Algebra of functions is $L^\infty(\Omega)$, the set of measurable bounded functions, which we shall endow with the norm
$\lVert f\rVert = \lVert f\rVert_\infty + \sup f-\inf f$.
Observe that convergence of measures in duality to $L^\infty(\Omega)$ is convergence in total variation.
For a transition kernel $\mchain{M}$, having an averaging operator $\op{L}_0$ with a spectral gap is a very strong condition, called \emph{uniform ergodicity} (the second term in the norm above is designed to get this equivalence). Under the (slightly stronger) contraction hypothesis, for any bounded measurable observable $\varphi$ Theorem \ref{theo:main-conc} thus yields for small enough $a$ an effective inequality of the form
\[\pr_\mu\Big[\lvert\hat\mu_n(\varphi)-\mu_0(\varphi)\rvert\ge a\Big]    \le 2.488 \exp\Big(- C\frac{n a^2}{\lVert\varphi\rVert_\infty^2} \delta_0  \Big)\]
where $\delta_0$ is the gap of the contraction of the Markov chain and $C$ is an absolute explicit constant.
Such explicit inequalities where obtained by Glynn and Ormoneit \cite{Glynn2002} and Kontoyiannis, Lastras-Montaño and Meyn \cite{Kontoyiannis-LMM} using the characterization of uniform ergodicity by the \emph{Doeblin minorization condition}; they obtain a non-optimal quadratic dependency on the gap (although their results are stated with another, directly related parameter $\beta$). More recently, an effective concentration inequality with the optimal dependency on $\delta_0$ and better constants than ours was obtained by Paulin \cite{Paulin2015} (Corollary 2.10). That result is stated in term of a certain mixing time, and for concentration around the expectation of $\hat\mu_n(\varphi)$; but it can be rephrased in term of the gap, and the bias $\hat\mu_n(\varphi)-\mu_0(\varphi)$ can easily be bounded. Dedeker and Gou\"ezel  \cite{dedeker2015subgaussian} proved concentration results (that can be made effective) under the more general hypothesis of \emph{geometric} ergodicity (they actually prove that geometric ergodicity is characterized by a subgaussian concentration inequality).

\subsection{Discrete hypercube}\label{sec:cube}

It is interesting to consider the same toy example as Joulin and Ollivier \cite{JO}, the lazy random walk on the discrete hypercube $\{0,1\}^N$: the transition kernel $\mchain{M}$ chooses uniformly a slot $i\in\{1,\dots, N\}$ and replaces it with the result of a fair coin toss.

We consider two kind of observables: $\frac1N$ Lipschitz maps such as the ``polarization'' $\rho$ giving the proportion of $1$'s in its argument, and the characteristic function $\one_S$ of a subset $S\subset \{0,1\}^N$. We shall distinguish further the case of a very regular set
$S=[0]:=\{(0,x_2,\dots,x_N) \colon x_i\in\{0,1\}\}$
and the case of ``scrambled'' sets, i.e. such that the dynamical variance of $\one_S$ is bounded by a constant independent of the dimension $N$; this is the case of sets such that every vertex has exactly $2pN$ neighbors with the same value of $\one_S$, where $p$ is fixed independently of $N$.

We compare our results with those of Joulin and Ollivier in Table \ref{tab:cube}. In the case of $\frac1N$-Lipschitz observable we apply Theorem \ref{theo:main-conc} with the weighted Lipschitz norm $\lVert\cdot \rVert_L := \lVert\cdot\rVert_\infty+N\Lip(\cdot)$; in the case of $\one_{[0]}$ we apply Theorem \ref{theo:main-conc} but with the ``local total variation'' norm
\[\lVert f\rVert_W:= \lVert f\rVert_\infty+\sup_{x\in\{0,1\}^N} \sum_{y\sim x} \lvert f(y)-f(x) \rvert\]
where $\sim$ denotes adjacency ($x\sim y$ whenever they differ in exactly one coordinate); in the case of $\one_S$ with a scrambled $S$, we apply Theorem \ref{theo:main-second} with the norm $\lVert\cdot\rVert_L$.
One sees that we obtain a weaker estimate in the case of $\rho$, but a better one in the case of $\one_{S}$, by exploiting the flexibility of our results in the choice of norm and in the possible use of bounds on the dynamical variance. The case of scrambled sets is notable, as we get a runtime independent of the dimension $N$.
\begin{table}[htp]
\centering
\begin{tabular}{r|ccc}
 & $\frac1N$-Lip maps & $\one_{[0]}$ & $\one_S$, scrambled $S$ \\
\hline
Joulin-Ollivier & $O\big(N+\frac{1}{a^2}\big)$ & $O\big(\frac{N^2}{a^2}\big)$ & $O\big(\frac{N^2}{a^2}\big)$ \\
Our best result & $O\big(\frac{N}{a^2}\big)$ & $O\big(\frac{N}{a^2}\big)$ & $O\big(\frac{1}{a^2}\big)$ 
\end{tabular}
\caption{Runtime to ensure error below $a \ll 1$ with good probability.}\label{tab:cube}
\end{table}

\subsection{Bernoulli convolutions and BV observables}
\label{sec:Bernoulli}

As a last example, let us consider the ``Bernoulli convolution'' of parameter $\lambda\in(0,1)$, defined as the law $\beta_\lambda$ of the random variable 
$\sum_{k\ge 1} \epsilon_k \lambda^k$
where the $\epsilon_k$ are independent variables taking the value $1$ with probability $1/2$ and the value $-1$ with probability $1/2$.

When $\lambda<1/2$, the support of $\beta_\lambda$ is a Cantor set of zero Lebesgue measure, so that $\beta_\lambda$ is singular (with respect to Lebesgue measure). When $\lambda=1/2$, $\beta_\lambda$ is the uniform measure on $[-1,1]$. But when $\lambda\in (1/2,1)$ (which we assume from now on), the question of the absolute continuity of $\beta_\lambda$ is very difficult, and fascinating. It was proved by Erd\"os \cite{Erdos1939} that if $\lambda$ is the inverse of a Pisot number, then $\beta_\lambda$ is singular, and a while later Solomyak discovered that for Lebesgue-almost all $\lambda$, $\beta_\lambda$ is absolutely continuous \cite{Solomyak1995}. See \cite{Peres2000} for more information on these questions.

One can realize $\beta_\lambda$ as the stationary law of the Markov transition kernel 
\[\mchain{M}=\big(m_x= \frac12 \delta_{T_0(x)}+ \frac12\delta_{T_1(x)}\big)_{x\in \mathbb{R}}\]
where $T_0(x) = \lambda x-\lambda$ and $T_1(x)=\lambda x+\lambda$.
In order to evaluate $\beta_\lambda(\varphi)$ by a MCMC method, one cannot use the methods developed for ergodic Markov chains since, conditionally to $X_0=x$, the law $m_x^k$ of $X_k$ is atomic and thus singular with respect to $\beta_\lambda$: $d_{\mathrm{TV}}(m_x^k,\beta_\lambda) = 1$ for all $k$. The convergence only holds for observables satisfying some regularity assumption, and it is natural to ask what regularity is needed.

Our results can deal with observables of bounded variation, a regularity which has the great advantage over e.g. Lipschitz to include the characteristic functions of intervals.
It can be proved that some iterate of $\mchain{M}$ is contracting on the space $\BV$ in the sense of Hypothesis \ref{hypo:L} (precisely, it is sufficient to iterate $\ell := \lfloor 1+ \log 2/\log \frac1\lambda\rfloor$ times). Applying Theorem \ref{theo:main-conc} to $(Y_k = X_{k\ell})_{k\ge 0}$ and setting $\hat\mu_n^Y = \frac1n \sum_{k=1}^n \delta_{Y_k}$ we get for any starting distribution $Y_0 \sim \mu$, any $\varphi\in\BV(I_\lambda)$, 
any positive $a <\lVert \varphi\rVert_{\BV}/3(2^{\ell+1}-1)$ and any $n\ge 120\cdot 2^\ell$:
\[ \pr_\mu\Big[\lvert\hat\mu_n^Y(\varphi)-\mu_0(\varphi)\rvert\ge a\Big]    \le 2.488 \exp\Big(- \frac{n a^2}{\lVert\varphi\rVert_{\BV}^2 (16.65 \cdot 2^\ell + 5.12)} \Big).\]

To the best of my knowledge, chains of this type together with BV observables could not be handled effectively by previously known results. For example \cite{Gomez-Dartnell2012} needs the observable to be at least $C^2$ to have explicit estimates, and they do not give a concentration inequality.

\section{Connection with perturbation theory}
\label{sec:perturbation}

To any $\varphi\in\fspace{X}$ (sometimes called a ``potential'' in this role) is associated a weighted averaging operator, called a transfer operator in the dynamical context:
\[\op{L}_\varphi f(x) = \int_\Omega  e^{\varphi(y)} f(y) \dd m_{x}(y).\]

The classical guiding idea for the present work combines two observations.
First, we have
\[\op{L}_\varphi^2 f(x_0) = \int_\Omega e^{\varphi(x_1)} \op{L}_\varphi f(x_1) \dd m_{x_0}(x_1) = \int_{\Omega\times\Omega} e^{\varphi(x_1)} e^{\varphi(x_2)} f(x_2) \dd m_{x_1}(x_2) \dd m_{x_0}(x_1)\]
and by a direct induction, denoting by $\dd m_{x_0}^n(x_1,\dots,x_n)$ the law of $n$ steps of a Markov chain following the transition $\mchain{M}$ and starting at $x_0$, we have
\[ \op{L}_\varphi^n f(x_0) = \int_{\Omega^n} e^{\varphi(x_1)+\dots+\varphi(x_n)} f(x_n) \dd m_{x_0}^n(x_1,\dots,x_n).\]
In particular, applying to the function $f=\one$, we get
\begin{equation*}
\op{L}_{\varphi}^n\one(x_0) = \int_{\Omega^n} e^{\varphi(x_1)+\dots+\varphi(x_n)} \dd m_{x_0}^n(x_1,\dots,x_n) = \esp_{x_0}\big[e^{\varphi(X_1)+\dots+\varphi(X_n)}\big]
\end{equation*}
where $(X_k)_{k\ge 0}$ is a Markov chain with transitions $\mchain{M}$ and the subscript on expectancy and probabilities specify the initial distribution ($x_0$ being short for $\delta_{x_0}$).

It follows by linearity that if the Markov chain is started with $X_0 \sim \mu$ where $\mu$ is any probability measure, then setting $\hat\mu_n\varphi := \frac1n\varphi(X_1)+\dots+\frac1n\varphi(X_n)$ we have
\begin{equation}
\esp_{\mu}\big[\exp(t\hat\mu_n\varphi)\big] = \int \op{L}_{\frac{t}{n}\varphi}^n\one(x) \dd\mu(x).
\label{eq:mgf}
\end{equation}
This makes a strong connection between the transfer operators and the behavior of $\hat\mu_n \varphi$.

Second, when the potential is small (e.g. $\frac tn \varphi$ with large $n$), the transfer operator is a perturbation of $\op{L}_0$, and their spectral properties will be closely related. This is the part that has to be made quantitative to obtain effective limit theorems.

We will state the perturbation results we need after introducing some notation. The letter $\op{L}$ will always denote a bounded linear operator, and $\lVert \cdot\rVert$ will be used both for the norm in $\fX$ and for the operator norm. From now on it is assumed that $\op{L}_0$ is a contraction with gap $\delta_0$. In \cite{K:perturbation} the leading eigenvalue of $\op{L}_0$ is denoted by $\lambda_0$, an eigenvector is denoted by $u_0$, and an eigenform (eigenvector of $\op{L}_0^*$) is denoted by $\phi_0$.

Two quantities appear in the perturbation results below. The first one is the \emph{condition number} $\tau_0 := \frac{\lVert \phi_0\rVert \lVert u_0\rVert}{\lvert \phi_0(u_0)\rvert}$.
To define the second one, we need to introduce $\pi_0$, the projection on $G_0$ along $\langle u_0\rangle$, which here writes $\pi_0(f)=f-\mu_0(f)$, and observe that by the contraction hypothesis  $(\op{L}_0-\lambda_0)$ is invertible when acting on $G_0$ (of course a spectral gap suffices). Then the \emph{spectral isolation} is defined as
\[\gamma_0 := \lVert (\op{L}_0-\lambda_0)_{|G_0}^{-1} \pi_0\rVert.\]

We shall denote by $\op{P}_0$ the projection on $\langle u_0\rangle$ along $G_0$, and set $\op{R}_0=\op{L}_0\circ \pi_0$. We then have the expression
\[\op{L}_0 = \lambda_0 \op{P}_0 + \op{R}_0\]
with $\op{P}_0\op{R}_0=\op{R}_0\op{P}_0=0$. This decomposition will play a role below, and can be done for all $\op{L}$ with a spectral gap: we denote by $\lambda_\op{L}, \pi_\op{L}, \op{P}_\op{L}, \op{R}_\op{L}$ the corresponding objects for $\op{L}$, and by $\lambda,\pi,\op{P}, \op{R}$ we mean the corresponding maps $\op{L}\mapsto \lambda_\op{L}$, etc.

Last, the notation $O_C(\cdot)$ is the Landau notation with an explicit constant $C$, i.e. $f(x)=O_C(g(x))$ means that for all $x$, $\lvert f(x)\rvert \le C \lvert g(x)\rvert$.

\begin{theo}[Theorems 2.3 and 2.6 and Proposition 5.1 (viii) of \cite{K:perturbation}] \label{theo:perturbation}
All $\op{L}$ such that $\displaystyle \lVert\op{L}-\op{L}_0\rVert < 1/(6\tau_0\gamma_0)$ have a simple isolated eigenvalue; $\lambda,\pi,\op{P},\op{R}$ are defined and analytic on this ball. 
Given any $K>1$, whenever $\displaystyle \lVert\op{L}-\op{L}_0\rVert \le (K-1)/(6K\tau_0\gamma_0)$  we have
\begin{align*}
\lambda_\op{L} &= \lambda_0 +O_{\tau_0+\frac{K-1}{3}} \big(\lVert\op{L}-\op{L}_0\rVert\big) \\
\lambda_\op{L} &= \lambda_0 + \phi_0(\op{L}-\op{L}_0)u_0 + O_{K \tau_0\gamma_0}\big(\lVert\op{L}-\op{L}_0\rVert^2\big) \\
\lambda_\op{L} &= \lambda_0 + \phi_0(\op{L}-\op{L}_0)u_0 + \phi_0(\op{L}-\op{L}_0)\op{S}_0(\op{L}-\op{L}_0)u_0 + O_{2 K^2\tau_0^2\gamma_0^2}\Big(\lVert \op{L}-\op{L}_0\rVert^3 \Big)
\\
\op{P}_\op{L} &= \op{P}_0+O_{2K\tau_0\gamma_0}(\lVert\op{L}-\op{L}_0\rVert) \\
\pi_\op{L} &= \pi_0+O_{\tau_0+\frac{K-1}{3}}(\lVert\op{L}-\op{L}_0\rVert)\\
 \Big\lVert D\Big[\frac{1}{\lambda} \op{R}\Big]_\op{L}\Big\rVert &\le  \frac{1}{\lvert \lambda_\op{L}\rvert}+\frac{\tau_0+\frac{K-1}{3}}{\lvert\lambda_\op{L}\rvert^2}\lVert\op{L}\rVert+2K\tau_0\gamma_0.
\end{align*}
\end{theo}

\begin{theo}[Corollary 2.12 from \cite{K:perturbation}]
\label{theo:perturbation-gap}
In the case $\lambda_0=\lVert \op{L}_0\rVert=1$, all $\op{L}$ such that
\[\lVert\op{L}-\op{L}_0\rVert \le \frac{\delta_0(\delta_0-\delta)}{6(1+\delta_0-\delta)\tau_0\lVert\pi_0\rVert}\]
have a spectral gap of size $\delta$ below $\lambda_\op{L}$, with constant $1$, i.e. for all $f$ on a closed hyperplane,  $\lVert \op{L}^n f \rVert \le \lvert \lambda_\op{L}\rvert^n(1 -\delta)^n \lVert f\rVert$.
\end{theo}

Since we will apply these results to the averaging operator $\op{L}_0$, we need to evaluate the parameters in this case.
\begin{lemm}
We have $\lambda_0=1$, $\tau_0=1$, $\lVert \pi_0\rVert \le 2$ and $\gamma_0\le 2/\delta_0$.
\end{lemm}

\begin{proof}
By the construction of $\op{L}_0$, we get $u_0=\one$ and $\lambda_0=1$; we mentioned that $\phi_0$ is identified with the stationary measure $\mu_0$.

By hypothesis $\lVert u_0\rVert =1$, and $\lVert \phi_0\rVert =1$ since $\lVert\cdot\rVert \ge \lVert\cdot \rVert_\infty$ and $\phi_0$ is a probability measure. Then $\lvert\phi_0(u_0)\rvert=\lvert \mu_0(\one)\rvert=1$ and it follows $\tau_0=1$.

Since for all $f\in\fspace{X}$, we have $\pi_0(f)=f-\mu_0(f)$ and $\lVert \mu_0(f) \one\rVert = \lvert \mu_0(f)\rvert \le \lVert f\rVert_\infty\le \lVert f\rVert$, we get $\lVert \pi_0\rVert \le 2$. (In general this trivial bound can hardly be improved without more information, notably on $\mu_0$: it may be the case that $\mu_0$ is concentrated on a specific region of the space, and then $f-\mu_0(f)$ could have norm close to twice the norm of $f$.)

Last, from the Taylor expansion $(1-\op{L}_0)^{-1} = \sum_{k\ge0} \op{L}_0^k$, the contraction with gap $\delta_0$, and the upper bound on $\lVert\pi_0\rVert$ we deduce $\gamma_0\le 2/\delta_0$.
\end{proof}

\section{Main estimates}
\label{sec:main}

Standing assumption \ref{hypo:L} ensures that for all small enough $\varphi$ we can apply the above perturbation results; recall that $\mu_0$ is the stationary measure, so that for all $f\in\fspace{X}$ we have $\int \op{L}_0 f \dd\mu_0=\int f\dd\mu_0$.

We will first apply Theorem \ref{theo:perturbation-gap} with $\delta=\delta_0/13$; this is somewhat arbitrary, but the exponential decay will be strong enough compared to other quantities that we don't need $\delta$ to be large. Taking it quite small allow for a larger radius where the result applies.

As a consequence of this choice, the following smallness assumption will often be needed:
\begin{equation} \lVert \varphi\rVert \le \log\Big(1+ \frac{\delta_0^2}{13+12\delta_0}\Big).
\label{eq:A1}
\end{equation}

We will often use $\varphi$ instead of $\op{L}_\varphi$ in subscripts: for example $\lambda_\varphi=\lambda_{\op{L}_\varphi}$ is the largest eigenvalue of $\op{L}_\varphi$, obtained by perturbation of $\lambda_0$, and $\pi_\varphi$ is linear projection on its eigenline along the stable complement appearing in the contraction hypothesis.
\begin{lemm}\label{lemm:estimates}
We have
$\op{L}_{\varphi}(\cdot) = \op{L}_0\big(\sum_{j\ge 0}\frac{\varphi^j}{j!} \cdot\big)$ and $\lVert \op{L}_{\varphi}-\op{L}_0\rVert  \le e^{\lVert \varphi\rVert}-1$.
If \eqref{eq:A1} holds, then we have
\begin{align*}
\lVert \op{L}_{\varphi}-\op{L}_0\rVert
  &\le \frac{\delta_0^2}{13+12\delta_0} \le \frac{1}{25} &
\op{L}_{\varphi} &= \op{L}_0 + O_{1.02}(\lVert \varphi\rVert) \\
&&  &= \op{L}_0+\op{L}_0(\varphi \cdot) + O_{0.507}(\lVert \varphi\rVert^2) \\
\lVert \pi_\varphi\rVert &\le 2.053 &  &= \op{L}_0\big( (1+\varphi+\frac12\varphi^2)\cdot\big) + O_{0.169}(\lVert \varphi\rVert^3).
\end{align*}
Assumption \eqref{eq:A1} is in particular sufficient to apply Theorem \ref{theo:perturbation-gap} with $\delta=\delta_0/13$ and Theorem \ref{theo:perturbation} with $K=1+12\delta_0/13$.
\end{lemm}

\begin{proof}
The first formula is a rephrasing of the definition of $\op{L}_{\varphi}$;
observe then that thanks to the assumption that $\fspace{X}$ is a Banach algebra, we have
\[\lVert \op{L}_{\varphi}-\op{L}_0\rVert = \lVert \op{L}_0\big((e^{\varphi}-1)\cdot\big) \rVert \le \lVert\op{L}_0\rVert \Big\lVert \sum_{j=1}^\infty \frac{\varphi^j}{j!}\Big\rVert \le \sum_{j=1}^\infty \frac{\lVert \varphi\rVert^j}{j!}  \le e^{\lVert \varphi\rVert}-1.\]

Observing that $x\mapsto x^2/(13+12 x)$ is increasing from $0$ to $1/25$ as $x$ varies from $0$ to $1$ completes the uniform bound of $\lVert \op{L}_{\varphi}-\op{L}_0\rVert$ and gives $\lVert \varphi\rVert \le \log(1+1/25):=b$. By convexity, we deduce that 
$e^{\lVert\varphi\rVert}-1\le (e^b-1) \frac{\lVert\varphi\rVert}{b} \le 1.02 \lVert\varphi\lVert$ and the zeroth order Taylor formula follows.

The higher-order estimates are obtained similarly:
\[\op{L}_\varphi = \op{L}_0\big((\one+\varphi + (e^\varphi-\varphi-1))\cdot\big) =\op{L}_0+\op{L}_0(\varphi\cdot) + O_{\lVert\op{L}_0\rVert}(e^\varphi-\varphi-1)\]
and using the triangle inequality, the convexity of $\frac{e^x-x-1}x$ and the bound on $\varphi$:
\[\lVert e^\varphi-\varphi-1\rVert \le \frac{e^{\lVert\varphi\rVert}-\lVert\varphi\rVert-1}{\lVert\varphi\rVert} \lVert\varphi\rVert \le \frac{e^b-b-1}{b^2} \lVert\varphi\rVert^2 \le 0.507 \lVert\varphi\rVert^2.\]
The second order remainder is bounded by
\[\lVert e^\varphi-\frac12\varphi^2-\varphi-1\rVert \le  \frac{e^b-\frac12 b^2-b-1}{b^3} \lVert\varphi\rVert^3 \le 0.169 \lVert\varphi\rVert^3\]
and finally, we have
\[\lVert \pi_\varphi\rVert \le \lVert \pi_0\rVert +\big(1+\frac{4\delta_0}{13}\big) \Vert \op{L}_\varphi-\op{L}_0\rVert \le 2+\big(1+\frac{4}{13}\big)\frac{1}{25}\le 2.053.\]
\end{proof}

\begin{lemm}\label{lemm:lambda}
Under \eqref{eq:A1} we have 
\begin{align*}
\lvert \lambda_{\varphi}-1\rvert &\le  0.0524 \qquad \lambda_{\varphi} = 1+ O_{1.334}(\lVert \varphi\rVert),\\
\lambda_{\varphi} &= 1+ \mu_0(\varphi) + O_{2.43+2.081 \delta_0^{-1}}(\lVert \varphi\rVert^2)
\end{align*}
and
\[\lambda_{\varphi}  = 1 + \mu_0(\varphi)+ \frac12\mu_0(\varphi^2) + \sum_{k\ge 1}\mu_0( \varphi \op{L}_0^k (\bar\varphi)) + O_{7.41+17.75\delta_0^{-1}+8.49\delta_0^{-2}}(\lVert\varphi\rVert^3).\]
\end{lemm}

\begin{proof}
With $K=1+12\delta_0/13$ we have $\tau_0+\frac{K-1}{3}=1+4\delta_0/13$ and by the Theorem \ref{theo:perturbation}, $\op{L}\mapsto\lambda_\op{L}$ has Lipschitz constant at most $1+4/13=17/13$. We get $\lvert \lambda_{\varphi}-\lambda_0\rvert 
  \le \frac{17}{13}\lVert\op{L}_{\varphi}-\op{L}_0\rVert$ from which we deduce both $\lvert \lambda_{\varphi}-1\rvert  \le \frac{17}{13\times 25} \le 0.0524  $ and $\lvert \lambda_{\varphi}-1\rvert  \le \frac{17}{13}1.02 \lVert \varphi\rVert \le 1.334\lVert \varphi\rVert$.

Now we use the first-order Taylor formula for $\lambda$, using $K\tau_0\gamma_0\le 2\delta_0^{-1}(1+12\delta_0/13)= \frac{24}{13}+2\delta_0^{-1}$:
\[\lambda_{\varphi}=1+\mu_0\big( (\op{L}_{\varphi}\one-\op{L}_0\one)\big)+O_{\frac{24}{13}+2\delta_0^{-1}}(\lVert \op{L}_{\varphi}-\op{L}_0\rVert^2),\]
then using $\op{L}_{\varphi}\one-\op{L}_0\one = \op{L}_0(\varphi)+O_{0.507}(\lVert \varphi\rVert^2)$ from Lemma \ref{lemm:estimates} we get
\[\mu_0( \op{L}_{\varphi}\one-\op{L}_0\one) = \mu_0( \op{L}_0(\varphi) ) + O_{0.507}(\lVert \varphi\rVert^2) = \mu_0(\varphi)+O_{0.507}(\lVert \varphi\rVert^2).\]
Using $\lVert \op{L}_{\varphi}-\op{L}_0\rVert \le  1.02\lVert \varphi\rVert$ gives the following constant in the final $O(\lVert \varphi\rVert^2)$ of the first-order formula:
\[0.507+(1.02)^2(\frac{24}{13}+2\delta_0^{-1}) \le 2.43+2.081 \delta_0^{-1}.\]

Then we apply the second-order Taylor formula:
\[\lambda_\varphi = 1+\mu_0(\op{L}_{\varphi}\one-\op{L}_0\one) + \mu_0\Big((\op{L}_\varphi-\op{L}_0) \op{S}_0(\op{L}_\varphi\one-\op{L}_0\one)\Big)+ O_{8K^2\delta_0^{-2}}(\lVert \op{L}_{\varphi}-\op{L}_0\rVert^3).\]
Using $\op{L}_{\varphi}\one-\op{L}_0\one = \op{L}_0(\varphi+\frac12 \varphi^2)+O_{0.169}(\lVert \varphi\rVert^3)$ from Lemma \ref{lemm:estimates} we first get
\begin{align*}
\mu_0(\op{L}_{\varphi}\one-\op{L}_0\one)
  &= \mu_0(\varphi)+\frac12 \mu_0(\varphi^2)+ O_{0.169}(\lVert \varphi\rVert^3).
\end{align*}
To simplify the second term, we recall that  $\op{L}_\varphi-\op{L}_0=\op{L}_0(\varphi \cdot)+O_{0.507}(\lVert \varphi\rVert^2)$ and $\op{S}_0 = (1-\op{L}_0)^{-1}\pi_0 = \big(\sum_{k\ge 0} \op{L}_0^k\big)\pi_0$
where $\pi_0$ is the projection on $\ker\mu_0$ along $\langle\one\rangle$, i.e. $\pi_0(f)=f-\mu_0(f)=:\bar f$, and has norm at most $2$. We thus have (noticing that in the second line both the main term and the remainder term  belong to $\ker\mu_0$):
\begin{align*}
\pi_0(\op{L}_{\varphi}\one-\op{L}_0\one) &= \pi_0 \big(\op{L}_0(\varphi)+O_{0.507}(\lVert\varphi\rVert^2)\big) = \op{L}_0(\bar\varphi) + O_{1.014}(\lVert\varphi\rVert^2) \\
\op{S}_0(\op{L}_\varphi \one-\op{L}_0\one) 
  &= \sum_{k\ge 1} \op{L}_0^k ( \bar\varphi)+O_{1.014 \delta_0^{-1}}(\lVert\varphi\rVert^2).
\end{align*}
We also have $\lVert \op{S}_0(\op{L}_\varphi \one-\op{L}_0\one)\rVert 
  \le  \frac{2}{\delta_0}\lVert \op{L}_\varphi \one-\op{L}_0\one \rVert 
  \le \frac{2.04}{\delta_0}\lVert \varphi\rVert$ and it comes
\begin{align*}
(\op{L}_\varphi-\op{L}_0)\op{S}_0(\op{L}_\varphi \one-\op{L}_0\one) 
  &= \op{L}_0\big(\varphi \sum_{k\ge 1} \op{L}_0^k (\bar\varphi)\big) +O_{1.014 \delta_0^{-1}}(\lVert\op{L}_\varphi-\op{L}_0\rVert \lVert\varphi\rVert^2) + \\
  & \qquad +O_{0.507}(\lVert\varphi\rVert^2\lVert \op{S}_0(\op{L}_\varphi \one-\op{L}_0\one)\rVert) \\
  &= \op{L}_0\big(\varphi \sum_{k\ge 1} \op{L}_0^k (\bar\varphi)\big) +O_{2.07\delta_0^{-1}}(\lVert\varphi\rVert^3) \\
\mu_0(\op{L}_\varphi-\op{L}_0)\op{S}_0(\op{L}_\varphi \one-\op{L}_0\one) 
  &= \sum_{k\ge 1}\mu_0(\varphi  \op{L}_0^k (\bar\varphi))+O_{2.07\delta_0^{-1}}(\lVert\varphi\rVert^3)
\end{align*}
where the reversal of sum and integral is enabled by normal convergence.

Last we observe $8K^2\delta_0^{-2}=8(\frac{12}{13}+\delta_0^{-1})^2 \le 6.82 + 14.77 \delta_0^{-1}+8\delta_0^{-2}$, and we gather all what precedes:
\begin{align*}
\lambda_\varphi 
 &= 1+\mu_0(\op{L}_{\varphi}\one-\op{L}_0\one) + \mu_0\Big((\op{L}_\varphi-\op{L}_0) \op{S}_0(\op{L}_\varphi\one-\op{L}_0\one)\Big)+ O_{8K^2\delta_0^{-2}}(\lVert \op{L}_{\varphi}-\op{L}_0\rVert^3) \\
  &= 1 + \mu_0(\varphi)+ \frac12\mu_0(\varphi^2) + O_{0.169}(\lVert \varphi\rVert^3) + \sum_{k\ge 1}\mu_0( \varphi \op{L}_0^k (\bar\varphi)) +O_{2.07\delta_0^{-1}}(\lVert\varphi\rVert^3) \\
  &\qquad\qquad + O_{(6.82 + 14.77 \delta_0^{-1}+8\delta_0^{-2})1.02^3}(\lVert\varphi\rVert^3)\\
  &=1 + \mu_0(\varphi)+ \frac12\mu_0(\varphi^2) + \sum_{k\ge 1}\mu_0( \varphi \op{L}_0^k (\bar\varphi)) + O_{7.41+17.75\delta_0^{-1}+8.49\delta_0^{-2}}(\lVert\varphi\rVert^3).
\end{align*}
\end{proof}

Under assumption \eqref{eq:A1}, we know that $\op{L}_{\varphi}$ is contracting with gap $\delta_0/13$, and we can write
$\op{L}_{\varphi} = \lambda_{\varphi} \op{P}_{\varphi} + \op{R}_{\varphi}$
where $\op{P}_{\varphi}$ is the projection to the eigendirection along the stable complement and $\op{R}_\varphi=\op{L}_\varphi\pi_\op{\varphi}$ is the composition of the projection to the stable complement and $\op{L}_\varphi$. Then it holds $\op{P}_\varphi\op{R}_\varphi=\op{R}_\varphi\op{P}_\varphi=0$, so that for all $n\in\mathbb{N}$:
\[\op{L}_{\varphi}^n = \lambda_{\varphi}^n \op{P}_{\varphi} + \op{R}_{\varphi}^n.\]

\begin{lemm}\label{lemm:RP}
Under assumption \eqref{eq:A1}, it holds
\begin{align*}
\big\lVert \Big(\frac{1}{\lambda_{\varphi}}\op{R}_{\varphi}\Big)^n\one \big\rVert 
  &\le (6.388+4.08\delta_0^{-1})(1-\delta_0/13)^{n-1} \lVert\varphi\rVert \\ 
\op{P}_\varphi \one &= \one+ O_{3.77+4.08\delta_0^{-1}}(\lVert \varphi\rVert).
\end{align*}
\end{lemm}

\begin{proof}
At any $\op{L}=\op{L}_\varphi$ where $\varphi$ satisfies \eqref{eq:A1} we have:
\begin{align*}
\big\lVert D\Big[\frac{1}{\lambda} \op{R}\Big]_\op{L}\big\rVert
  &\le \frac{1}{\lvert \lambda_\op{L}\rvert}+\frac{17/13}{\lvert\lambda_\op{L}\rvert^2}\lvert\op{L}\rvert+2K\tau_0\gamma_0\\
  &\le \frac{1}{0.9476}+\frac{17}{13\times 0.9476^2}\times 1.04+\frac{48}{13}+\frac{4}{\delta_0}
  &&\le 6.263+\frac{4}{\delta_0}
\end{align*}
so that
\begin{align*}
\big\lVert\frac{1}{\lambda_{\varphi}} \op{R}_{\varphi}\one\big\rVert
&= \big\lVert\frac{1}{\lambda_{\varphi}} \op{R}_{\varphi}\one-\frac{1}{\lambda_0} \op{R}_0\one\big\rVert
 \le (6.263+\frac{4}{\delta_0})\lVert\op{L}_{\varphi}-\op{L}_0\rVert \lVert \one\rVert \\
 &\le (6.388+4.08 \delta_0^{-1})\lVert \varphi\rVert.
\end{align*}
Moreover since $\op{R}_\op{L}$ takes its values in $G_\op{L}$ where $\pi_\op{L}$ acts as the identity, we have
$\lVert\op{R}_{\varphi}^n\one\rVert \le \lambda_{\varphi}^{n-1} (1-\delta_0/13)^{n-1} \lVert \op{R}_\op{L}\one\rVert$
from which the first inequality follows.

Then we have $\op{P}_\varphi = \op{P}_0 + O_{2K\tau_0\gamma_0}(\lVert \op{L}_\varphi-\op{L}_0\rVert)$, which yields the claimed result using $K=1+12\delta_0/13$, $\tau_0=1$, $\gamma_0\le 2\delta_0^{-1}$ and $\lVert \op{L}_\varphi-\op{L}_0\rVert \le 1.02 \lVert\varphi\rVert$.
\end{proof}

This control of $\op{P}_\varphi$ and $\op{R}_\varphi$ can be then be used to reduce the estimation of $\op{L}_\varphi^n\one$ to the estimation of $\lambda_\varphi^n$.
\begin{coro}\label{coro:L}
Under assumptions \eqref{eq:A1} and
\begin{equation}
n\ge 1+\frac{\log 100}{-\log(1-\delta_0/13)}
\label{eq:bound-n0}
\end{equation}
it holds
\begin{align*} \op{L}_{\varphi}^n\one &= \lambda_\varphi^n \big( 1+O_{3.834+4.121\delta_0^{-1}}(\lVert\varphi\rVert)\big) \\
\lambda_\varphi^n &= \exp\big(n \mu_0(\varphi) +O_{3.36+2.081 \delta_0^{-1}}(n\lVert\varphi\rVert^2)\big) \\
\lambda_\varphi^n  &=\exp\big(n \mu_0(\varphi) + \frac12 n\sigma^2(\varphi) + O_{10.89+ 20.04\delta_0^{-1} + 8.577 \delta_0^{-2}}(n\lVert \varphi\rVert^3) \big). 
\end{align*}
\end{coro}

\begin{proof}
Assuming \eqref{eq:A1}, Lemma \ref{lemm:RP} yields $\op{L}_{\varphi}^n\one = \lambda_{\varphi}^n \op{P}_{\varphi}\one+\op{R}_{\varphi}^n\one = \lambda_{\varphi}^n A$ where
\begin{equation}
A := \one+O_{3.77+4.08\delta_0^{-1}}(\lVert \varphi\rVert) + O_{6.388+4.08\delta_0^{-1}}\big(\big(1-\frac{\delta_0}{13}\big)^{n-1}\lVert \varphi \rVert\big) \label{eq:lambda^n}
\end{equation}
is easily controlled if we ask \eqref{eq:bound-n0}, under which we have
\[A  =1+O_{3.77+4.08\delta_0^{-1}}(\lVert \varphi\rVert) + O_{0.064+0.041\delta_0^{-1}}(\lVert \varphi \rVert)
  = 1+O_{3.834+4.121\delta_0^{-1}}(\lVert\varphi\rVert).\]

The first estimate for $\lambda_\varphi^n$ is obtained through the first-order Taylor formula. We use the monotony and convexity of $x\mapsto(\log(1+x)-x)/x$ and set
$x =\lambda_\varphi-1 \in [-b,b]$ with $b=0.0524$ to evaluate $\log(\lambda_\varphi)$:
\begin{align*}
\Big\lvert \frac{\log(1+x)-x}{x} \Big| &\le \frac{\log(1-b)+b}{-b} \cdot \frac{\lvert x\rvert}{b} \le 0.52 \lvert x\rvert \\
\log(\lambda_\varphi) &=  \lambda_\varphi-1+O_{0.52}(\lvert \lambda_\varphi-1\rvert^2)=  \lambda_\varphi-1+O_{0.52\times 1.334^2}(\lVert \varphi\rVert^2) \\
  &=  \lambda_\varphi-1+O_{0.926}(\lVert \varphi\rVert^2).
\end{align*}
and then using $\lambda_{\varphi} = 1+ \mu_0(\varphi) + O_{2.43+2.081 \delta_0^{-1}}(\lVert \varphi\rVert^2)$ from Lemma \ref{lemm:lambda}:
\begin{align*}
\lambda_{\varphi}^n
  &= \exp\big(n\log(\lambda_{\varphi})\big) = \exp\big(n(\lambda_{\varphi}-1)+O_{0.926}(n\lVert\varphi\rVert^2)\big)\\
  &= \exp\big(n \mu_0(\varphi) +O_{3.36+2.081 \delta_0^{-1}}(n\lVert\varphi\rVert^2)\big).
\end{align*}

The second estimate for $\lambda_\varphi^n$ is obtained, of course, from the second-order formula given in Lemma \ref{lemm:lambda}:
\[\lambda_{\varphi}  = 1 + \mu_0(\varphi)+ \frac12\mu_0(\varphi^2) + \sum_{k\ge 1}\mu_0( \varphi \op{L}_0^k (\bar\varphi)) + O_{7.41+17.75\delta_0^{-1}+8.49\delta_0^{-2}}(\lVert\varphi\rVert^3).\]
Here, it is somewhat tedious to use a convexity argument and we instead use the slightly less precise Taylor formula: for $x\in [-b,b]$ (where again $b=0.0524$) we have 
\[\Big\lvert\frac16\frac{\dd^3}{\dd x^3}\log(1+x)\Big\rvert\le \frac{2}{6(1-0.0524)^3}\le 0.392\]
so that 
\[\log(1+x) = x-\frac{1}{2}x^2 + O_{0.392}(x^3) \]
and therefore (using at one step $\lvert \mu_0(\varphi)\rvert\le \lVert\varphi\rVert$):
\begin{align*}
\log(\lambda_\varphi) &= (\lambda_\varphi-1) - \frac12(\lambda_\varphi-1)^2 + O_{0.392}((\lambda_\varphi-1)^3) \\
  &= \mu_0(\varphi) + \frac12\mu_0(\varphi^2) + \sum_{k\ge 1} \mu_0(\varphi \op{L}_0^k \bar\varphi) + O_{7.41+17.75\delta_0^{-1}+8.49\delta_0^{-2}}(\lVert\varphi\rVert^3) \\
  &\qquad -\frac12\big(\mu_0(\varphi)+O_{2.43+2.081 \delta_0^{-1}}(\lVert \varphi\rVert^2)\big)^2+O_{0.392\times 1.334^3}(\lVert \varphi\rVert^3)\\
  &= \mu_0(\varphi) + \frac12 \sigma^2(\varphi) + O_{10.771+19.831\delta_0^{-1}+8.49\delta_0^{-2}}(\lVert \varphi\rVert^3) \\
  & \qquad +O_{2.953+5.06\delta_0^{-1}+2.166\delta_0^{-2}}(\lVert \varphi\rVert^4).
\end{align*}
Now assumption \eqref{eq:A1} ensures $\lVert\varphi\rVert\le 0.04$, so that we can combine the two error terms into $O_{c}(\lVert \varphi\rVert^3)$ with
$c = 10.771+19.831\delta_0^{-1}+8.49\delta_0^{-2} + 0.04 (2.953+5.06\delta_0^{-1}+2.166\delta_0^{-2})
 \le  10.89+ 20.04\delta_0^{-1} + 8.577 \delta_0^{-2}$.
\end{proof}

\section{Concentration inequalities}
\label{sec:concentration}

We will in this section apply Corollary \ref{coro:L} to $\frac{t}{n}\varphi$ instead of $\varphi$, which we can do as soon as  $n$ is large enough with respect to $t$ and $\lVert\varphi\rVert$ in the sense that
\begin{equation}
n\ge \frac{\lVert t\varphi\rVert}{\log\Big(1+\frac{\delta_0^2}{12+13\delta_0}\Big)} \quad\mbox{and}\quad n\ge 1+\frac{\log 100}{-\log(1-\delta_0/13)}.
\label{eq:small-t}
\end{equation}
(These conditions can be replaced by the stronger but simpler conditions $n\ge 26\frac{\lVert t\varphi\rVert}{\delta_0^2}$ and $n \ge \frac{60}{\delta_0}$, respectively.)

Under conditions \eqref{eq:small-t}, we obtain our first control of the moment generating function of the empiric mean 
$\hat\mu_n(\varphi) := \frac{1}{n}\varphi(X_1)+\dots+\frac{1}{n}\varphi(X_n)$
by plugging the first-order estimate of Corollary \ref{coro:L} in \eqref{eq:mgf}:
\begin{align*}
\frac{\esp_\mu\big[\exp(t\hat\mu_n(\varphi))\big]}{\exp(t \mu_0(\varphi))} &= e^{-t\mu_0(\varphi)}\int \op{L}_{\frac{t}{n}\varphi}^n\one(x) \dd\mu(x) \\
  &= \big(1+O_{3.834+4.121\delta_0^{-1}}(\frac{t}{n}\lVert\varphi\rVert)\big)\exp(O_{3.36+2.081 \delta_0^{-1}}(\frac{t^2}{n}\lVert\varphi\rVert^2)).
\end{align*}

By the classical Chernov bound, it follows that for all $a,t>0$:
\begin{multline}
\pr_\mu\big[\lvert \hat\mu_n(\varphi)-\mu_0(\varphi)\rvert\ge a\big] \\ 
  \le  \big(2+(7.668+8.242\delta_0^{-1})\frac{t}{n}\lVert\varphi\rVert\big) \exp\big(-at+(3.36+2.081 \delta_0^{-1})\frac{t^2}{n}\lVert\varphi\rVert^2)\big). \label{eq:Chernov}
\end{multline}

\subsection{Gaussian regime}

Our first concentration inequality is obtained by choosing $t$ to optimize the argument of the exponential in \eqref{eq:Chernov}, i.e. taking
\[t=\frac{na}{2(3.36+2.081\delta_0^{-1})\lVert\varphi\rVert^2}.\]
This choice can be made as soon as $a$ is small enough: indeed the first condition on $n$ then reads
\[a\le (6.72+4.162\delta_0^{-1})\log\Big(1+\frac{\delta_0^2}{12+13\delta_0}\Big)\lVert \varphi\rVert =: a_{\text{max}} \lVert \varphi\rVert.\]
Let us find a simpler lower bound for the right-hand side:
\[a_{\text{max}} \ge (6.72+4.162\delta_0^{-1})\cdot 0.98 \frac{\delta_0^2}{12+13\delta_0} \ge \frac{6.58 \delta_0+4}{13\delta_0+12}\delta_0 
  \ge \frac{\delta_0}{3}
\]
so that a sufficient condition to make the above choice for $t$ is
\begin{equation}
a\le \frac{\delta_0\lVert \varphi\rVert}{3}.\label{eq:a}
\end{equation}
Then the argument in the exponential becomes
\begin{equation*}
-at+(3.36+2.081 \delta_0^{-1})\frac{t^2}{n}\lVert\varphi\rVert^2
  \le -\frac{n a^2}{(13.44+8.324\delta_0^{-1})\lVert\varphi\rVert^2}
\end{equation*}
and the constant in front:
\begin{align*}
2+(7.668+8.242\delta_0^{-1})\frac{t}{n}\lVert\varphi\rVert
  &\le 2+\frac{(7.668+8.242\delta_0^{-1})a}{(6.72+4.162\delta_0^{-1})\lVert\varphi\rVert}\\
  &\le 2+ \frac{7.668\delta_0^2+8.242\delta_0}{20.16\delta_0+12.486}\\
  &\le 2+\frac{7.668+8.242}{20.16+12.486} \le 2.488,
\end{align*}
which is the first part of Theorem \ref{theo:main-conc}
(one can also bound the front constant in a different way to show it can be taken close to $2$ for small $a$).

\subsection{Exponential regime}

For larger $a$, we obtain a result with exponential decay by taking $t$ as large as allowed by the first smallness condition \eqref{eq:small-t}, i.e. $t\simeq \frac{n}{\lVert\varphi\rVert} \log\Big(1+\frac{\delta_0^2}{12+13\delta_0}\Big)$.
To simplify, we precisely take the slightly smaller
\[t=\frac{n}{\lVert\varphi\rVert} \cdot \frac{0.98\delta_0^2}{12+13\delta_0}. \]
Then the argument in the exponential becomes
\begin{multline*}
-at+(3.36+2.081 \delta_0^{-1})\frac{t^2}{n}\lVert\varphi\rVert^2) \\
  = n\frac{0.98 \delta_0^2}{12+13\delta_0}\Big(-\frac{a}{\lVert\varphi\rVert}+\frac{0.98(3.36\delta_0^2+2.081\delta_0)}{12+13\delta_0}\Big) \\
  \le -n\frac{0.98 \delta_0^2}{12+13\delta_0}\Big(\frac{a}{\lVert\varphi\rVert}-0.254\delta_0\Big)
\end{multline*}
and the constant in front:
\begin{align*}
2+(7.668+8.242\delta_0^{-1})\frac{t}{n}\lVert\varphi\rVert
  &= 2+(7.668+8.242\delta_0^{-1})\frac{0.98\delta_0^2}{12+13\delta_0}\\
  &= 2+ \frac{7.515\delta_0^2+8.078\delta_0}{12+13\delta_0}\\
  &\le 2+\frac{15.593}{25} \le 2.624
\end{align*}
and we obtain the second part of Theorem \ref{theo:main-conc}.

\subsection{Second-order concentration}\label{sec:second-order}

In the case one has a good upper bound for the dynamical variance $\sigma^2(\varphi)$
then the previous concentration results can be improved by using the second-order formula in Corollary \ref{coro:L}, which yields
\begin{multline*}
\frac{\esp_\mu\big[\exp(t\hat\mu_n(\varphi))\big]}{\exp(t \mu_0(\varphi))} = \exp\Big(\frac{t^2}{2n}\sigma^2(\varphi)+ O_{10.89 + 20.04\delta_0^{-1} + 8.577 \delta_0^{-2}}\big(\frac{t^3}{n^2}\lVert \varphi\rVert^3\big)\Big) \\ \times \big(1+O_{3.834+4.121\delta_0^{-1}}\big(\frac{t}{n}\lVert\varphi\rVert\big)\big)
\end{multline*}
so that, if we know $\sigma^2(\varphi)\le U$:
\begin{multline*}
\pr_\mu\big[\lvert \hat\mu_n(\varphi)-\mu_0(\varphi)\rvert\ge a\big]
  \le  \Big(2+\frac{(7.668+8.242\delta_0^{-1})t}{n}\lVert\varphi\rVert\Big) \\ \times \exp\big(-at+\frac{t^2}{2n}U+ C\frac{t^3}{n^2} \lVert \varphi\rVert^3\big) 
\end{multline*}
where $C$ can be any number above $10.89 + 20.04\delta_0^{-1} + 8.577 \delta_0^{-2}$. To get a compact expression, we observe that $0.89+0.04\delta_0^{-1}\le 0.93\delta_0^{-2}$  so that
\[10.89 + 20.04\delta_0^{-1} + 8.577 \delta_0^{-2} \le 10+20\delta_0^{-1} + 9.507 \delta_0^{-2} \le 10(1+\delta_0^{-1})^2 =:C.\]

The choice of $t$ can then be adapted to the circumstances; we will only explore the choice $t=an/U$ which is nearly optimal when $a$ is small.

This choice can be made as soon as
\begin{equation*}
a \le \frac{U}{\lVert\varphi\rVert} \log\Big(1+\frac{\delta_0^2}{12+13\delta_0}\Big)
\end{equation*}
and entails the following upper bound for the front constant:
\[2+(7.668+8.242\delta_0^{-1})\frac{\delta_0^2}{12+13\delta_0} \le 2+\frac{7.668+8.242}{12+13} \le 2.637.\]
Meanwhile, the exponent becomes
\begin{align*}
-at+\frac{t^2}{2n}U+C\frac{t^3}{n^2}\lVert\varphi\rVert^3
  &= -\frac{a^2n}{2U}+\frac{C\lVert\varphi\rVert^3a^3n}{U^3}
\end{align*}
yielding Theorem \ref{theo:main-second}.

\section{Berry-Esseen bounds}\label{sec:BE}

In this section, we use the second-order Taylor formula for the leading eigenvalue to prove effective Berry-Esseen bounds. The method we use is the one proposed by Feller \cite{Feller}, which does not yield the best constant in the IID case, but is quite easily adapted to the Markov or dynamical case as observed in  \cite{Coelho-Parry}.

The starting point is a ``smoothing'' argument that allows to translate the proximity of characteristic functions into a proximity of distribution functions.

\begin{prop}[\cite{Feller}] \label{prop:Feller}
Let $F,G$ be the distribution functions and $\phi,\gamma$ be the characteristic functions of real random variables with vanishing expectation. Assume $G$ is differentiable and $\lVert G'\rVert_\infty \le m$; then for all $T>0$:
\[\lVert F-G\rVert_\infty \le \frac{1}{\pi} \int_{-T}^T \Big\lvert \frac{\phi(t)-\gamma(t)}{t} \Big\lvert \dd t + \frac{24 m}{\pi T}.\]
\end{prop}

We set $G(T)=(2\pi)^{-\frac12}\int_{-\infty}^T e^{-\frac{t^2}{2}} \dd t$ the reduced normal distribution function (so that $\lVert G'\rVert_\infty =(2\pi)^{-\frac12}$) and $\gamma(t)=e^{-\frac{t^2}{2}}$, and apply the above estimate to the distribution function $F_n$ of the random variable $Y_n = \frac{1}{\sqrt{n}}(\tilde\varphi(X_1)+\dots+\tilde\varphi(X_n))$, where here $\tilde\varphi$ is the fully normalized version of $\varphi$:
\[\tilde\varphi = \frac{\varphi-\mu_0(\varphi)}{\sigma(\varphi)} \qquad\mbox{where } \sigma^2(\varphi)=\mu_0(\varphi^2) - (\mu_0 \varphi)^2 + 2\sum_{k\ge 1} \mu_0\big(\varphi  \op{L}_0^k (\bar\varphi)\big),\]
assuming $\sigma^2(\varphi)> 0$ and with $\bar\varphi := \varphi-\mu_0(\varphi)$. The point is then to use the spectral method to obtain an expression of the characteristic function $\phi_n$ of $Y_n$ close to the expression of $\gamma$.

We start by showing that the norm of a normalized potential is bounded away from zero.
\begin{lemm}\label{lemm:away0}
We have $\lVert \tilde \varphi\rVert \ge \sqrt{\delta_0/2}$.
\end{lemm}

\begin{proof}
We have $\sigma^2(\varphi) =\sigma^2(\bar\varphi)
  \le \lVert \bar\varphi^2\rVert_\infty + 2\sum_{k\ge1}\lVert \bar \varphi\rVert_\infty (1-\delta_0)^k \lVert \bar \varphi\rVert
  \le \lVert \bar\varphi\rVert^2\big(\frac{2}{\delta_0}-1 \big).$
Using $\sigma^2(\tilde\varphi)=1$ we get $\lVert \tilde \varphi\rVert \ge \big(\frac{2}{\delta_0}-1 \big)^{-\frac12}$ and the result follows.
\end{proof}

This has a first interesting consequence: if assumption \eqref{eq:bound-n0} is not satisfied, we have in particular $n \le 60/\delta_0$ and Lemma \ref{lemm:away0} implies that in the conclusion of Theorem \ref{theo:main-BE} the right-hand side is (much) larger than $1$, making the conclusion vacuously true (the left-hand side is always less than $1$). It follows that we only need to consider the case when \eqref{eq:bound-n0} is satisfied even though we did not include it in the hypotheses. For the same reason, we can and do assume $n\ge 10\,000$.

To apply the estimates from Section \ref{sec:main} to $\frac{it}{\sqrt{n}}\tilde\varphi$, it is therefore sufficient to have
\begin{equation}
\sqrt{n}\ge \frac{\lVert t\tilde \varphi\rVert}{\log\Big(1+\frac{\delta_0^2}{13+12\delta_0} \Big)}.
\label{eq:assumptions-recall}
\end{equation}

\begin{lemm}\label{lemm:BE-estimates1}
Under assumption \eqref{eq:assumptions-recall} we have
\begin{align*}
\phi_n(t) &= \lambda_{\frac{it}{\sqrt{n}}\tilde\varphi}^n \Big(1+O_{3.668+4.121\delta_0^{-1}}(\lVert \frac{t}{\sqrt{n}} \tilde\varphi\rVert \Big) \\
\lambda_{\frac{it}{\sqrt{n}}\tilde\varphi}^n 
  &= \exp\Big(-\frac{t^2}{2}  + O_{10.89+20.04\delta_0^{-1}+8.577\delta_0^{-2}} (\frac{1}{\sqrt{n}}\lVert t \tilde\varphi\rVert^3)\Big).
\end{align*}
\end{lemm}

\begin{proof}
Applying formula \eqref{eq:mgf} to $\frac{it}{\sqrt{n}}\tilde\varphi$, we obtain the following expression for the characteristic function (where $\mu$ is the law of $X_0$):
\[\phi_n(t) = \int \op{L}_{\frac{it}{\sqrt{n}}\tilde\varphi}\one(x) \dd\mu(x) 
= \lambda_{\frac{it}{\sqrt{n}}\tilde\varphi}^n \underbrace{\Big( \int \op{P}_{\frac{it}{\sqrt{n}}\tilde\varphi}\one \dd\mu + \int \big[\op{R}/\lambda\big]_{\frac{it}{\sqrt{n}}\tilde\varphi}^n\one \dd\mu \Big)}_{=: A}\]

Corollary \ref{coro:L} gives the claimed expression for $\lambda_{\frac{it}{\sqrt{n}}\tilde\varphi}^n$ and
\[A=\int \op{P}_{\frac{it}{\sqrt{n}}\tilde\varphi}\one \dd\mu + \lambda_{\frac{it}{\sqrt{n}}\tilde\varphi}^{-n}\int \op{R}_{\frac{it}{\sqrt{n}}\tilde\varphi}^n\one \dd\mu = 1+O_{3.668+4.121\delta_0^{-1}}(\lVert \frac{t}{\sqrt{n}} \tilde\varphi\rVert).\]
\end{proof}

\begin{lemm}\label{lemm:majo-BE}
Under assumption \eqref{eq:assumptions-recall}, for any $\alpha\in(0,0.5)$, if
\begin{equation}
\sqrt{n} \ge \frac{10.89+20.04\delta_0^{-1}+8.577\delta_0^{-2}}{0.5-\alpha}  \lvert t\rvert \lVert\tilde \varphi\rVert^3
\label{eq:cond-BE}
\end{equation}
then
\begin{align}
\lvert \phi_n(t)-\gamma(t)\rvert  &\le 1.32 n e^{-0.9999\alpha t^2} \big\lvert \phi_n(t)^{\frac{1}{n}}-\gamma(t)^{\frac{1}{n}}\big\rvert.
\label{eq:majo-BE}
\end{align}
\end{lemm}

\begin{proof}
Following Feller \cite{Feller}, we use that
for all $a,b,c$ with
$\lvert a\rvert, \lvert b\rvert \le c$ and all $n\in\mathbb{N}$:
\begin{equation}
\lvert a^n-b^n \rvert \le n\lvert a-b\rvert c^{n-1}.
\label{eq:Feller}
\end{equation}
We take $a=\phi_n(t)^{\frac{1}{n}}$, $b=\gamma(t)^{\frac{1}{n}}$ and $c$ an upper bound which we will now choose.
Feller takes $c=e^{-\frac{t^2}{4n}}$, but we need two adaptations and take $c=1.32^{\frac1n}e^{-\alpha\frac{t^2}{n}}$ where $\alpha \in(0,0.5)$ will be optimized later on. We already have $\gamma(t)^{\frac{1}{n}} = e^{-\frac{t^2}{2n}}\le c$ and need to ensure the same bound for $\phi_n$.
We have
\[\phi_n(t)^{\frac{1}{n}} \le e^{-\frac{t^2}{2n}}\exp\big((10.89+20.04\delta_0^{-1}+8.577\delta_0^{-2})(\frac{1}{n^{3/2}}\lVert t \tilde\varphi\rVert^3)\big)A^{\frac1n}\]
where, using $\lVert \frac{t}{\sqrt{n}} \tilde\varphi\rVert \le \frac{\delta_0^2}{13+12\delta_0}$,
\[A \le 1+(3.834+4.121\delta_0^{-1}) \lVert \frac{t}{\sqrt{n}} \tilde\varphi\rVert \le 1.32.\]
To ensure $\phi_n(t)^{\frac1n}\le c$, it is therefore sufficient that
\[(10.89+20.04\delta_0^{-1}+8.577\delta_0^{-2})(\frac{1}{\sqrt{n}}\lVert t \tilde\varphi\rVert^3) \le (0.5-\alpha)t^2,\]
i.e. Condition \eqref{eq:cond-BE} suffices. 
Using $n\ge 10\,000$ to bound $(n-1)/n$ by $0.9999$ in \eqref{eq:Feller}, we then obtain \eqref{eq:majo-BE}.
\end{proof}

\begin{lemm}\label{lemm:majo-BE2}
Under assumption \eqref{eq:assumptions-recall} we have
\[\big\lvert \phi_n(t)^{\frac{1}{n}}-\gamma(t)^{\frac{1}{n}}\big\rvert
\le \frac{ f \lVert t \tilde\varphi\rVert^3 + g \lVert t\tilde\varphi\rVert}{n^{3/2}} +\frac{t^4}{8n^2}\]
with $f=7.41+17.75\delta_0^{-1}+8.49\delta_0^{-2}$ and $g=4.036+4.338\delta_0^{-1}$
\end{lemm}

\begin{proof}
We follow Feller again and write
\begin{align}
\big\lvert \phi_n(t)^{\frac{1}{n}}-\gamma(t)^{\frac{1}{n}}\big\rvert
   &\le \Big\lvert \lambda_{\frac{it}{\sqrt{n}}\tilde\varphi}A^{\frac1n} -1+\frac{t^2}{2n}\Big\rvert + \Big\lvert e^{-\frac{t^2}{2n}} -1+\frac{t^2}{2n}\Big\rvert.
   \label{eq:Feller2}
\end{align}
where $A$ is defined in the proof of Lemma \ref{lemm:BE-estimates1}.
Since for all $x\in [0,+\infty\mathclose[$ we have $0\le e^{-x}-1+x\le \frac12 x^2$, the second summand is bounded above by 
$\frac{t^4}{8n^2}$. To deal with the first summand we start by a finer evaluation of $A$:
\[
A^{\frac1n} = (1+O_{3.834+4.121\delta_0^{-1}}(\lVert \frac{t}{\sqrt{n}} \tilde\varphi\rVert))^{\frac{1}{n}} 
  \le \exp\big(\frac{1}{n^{3/2}}(3.834+4.121\delta_0^{-1}) \lVert t \tilde\varphi\rVert)\big).\]
By our assumptions the argument of the exponential is not greater than
\[
\frac{1}{n}(3.834+4.121\delta_0^{-1}) \log\Big(1+\frac{\delta_0^2}{13+12\delta_0}\Big)
  \le \frac{1}{10\,000}\frac{3.834 \delta_0^2+4.121\delta_0}{13+12\delta_0}\\
  \le 0.0001.
\]
Using $e^{0.0001}\le 1.00011$, for all $\varepsilon\in [0,0.0001]$ we have $\exp(\varepsilon)\le 1+ 1.00011\varepsilon$ so that:
\[
A^{\frac{1}{n}}
\le 1+ \frac{3.835+4.122\delta_0^{-1}}{n^{3/2}}\lVert t\tilde\varphi\rVert.
\]
Using Lemma \ref{lemm:lambda}, definition of $\sigma^2$ and normalization of $\tilde\varphi$, we have:
\[\lambda_{\frac{it}{\sqrt{n}}\tilde\varphi}=1 - \frac{t^2}{2n} + O_{7.41+17.75\delta_0^{-1}+8.49\delta_0^{-2}}\big(\lVert\frac{t}{\sqrt{n}}\tilde\varphi\rVert^3\big).\]
The lower order terms simplify in the first summand of \eqref{eq:Feller2} and we obtain
\begin{multline*}
\big\lvert \phi_n(t)^{\frac{1}{n}}-\gamma(t)^{\frac{1}{n}}\big\rvert \\
   \le \Big\lvert O_{7.41+17.75\delta_0^{-1}+8.49\delta_0^{-2}}(\lVert \frac{t}{\sqrt{n}} \tilde\varphi\rVert^3) + \lambda_{\frac{it}{\sqrt{n}}\tilde\varphi}\frac{3.835+4.122\delta_0^{-1}}{n^{3/2}}\lVert t\tilde\varphi\rVert \Big\rvert  +\frac{t^4}{8n^2} \\
   \le \frac{f\lVert t \tilde\varphi\rVert^3 + g\lVert t\tilde\varphi\rVert}{n^{3/2}}  +\frac{t^4}{8n^2}
\end{multline*}
(using $g\ge 1.0524(3.835+4.122\delta_0^{-1})$).
\end{proof}

For all $T>0$ such that the above conditions \eqref{eq:assumptions-recall} and \eqref{eq:cond-BE} hold for all $t\in[-T,T]$, we have by Proposition \ref{prop:Feller} and Lemmas \ref{lemm:majo-BE}, \ref{lemm:majo-BE2}:
\begin{align}
\lVert F_n-G\rVert_\infty &\le \frac{1}{\pi} \int_{-T}^T \Big\lvert \frac{\phi(t)-\gamma(t)}{t} \Big\lvert \dd t + \frac{24 m}{\pi T} \nonumber\\
  &\le \frac{2.64}{\pi}\int_0^T \frac{n}{t} e^{-0.9999\alpha t^2} \big\lvert \phi_n(t)^{\frac{1}{n}}-\gamma(t)^{\frac{1}{n}}\big\rvert \dd t + \frac{3.048}{T} \nonumber\\
  &\le \frac{2.64}{\pi\sqrt{n}}\int_0^\infty e^{-0.9999\alpha t^2} \big(f\lVert \tilde\varphi\rVert^3 t^2 + g \lVert \tilde\varphi\rVert+ht^3\big) \dd t + \frac{3.048}{T}
\nonumber
\end{align}
where $f$, $g$ are defined in Lemma \ref{lemm:majo-BE2} and, using $n\ge 10\,000$, $h = 0.00125$. We want to take $T$ as large as possible to lower the last term, but we need to ensure conditions \eqref{eq:assumptions-recall} and \eqref{eq:cond-BE}, i.e.:
\[T \le \frac{\sqrt{n}}{\lVert \tilde \varphi\rVert}\log\Big(1+\frac{\delta_0^2}{13+12\delta_0} \Big)  \quad\mbox{and}\quad  T \le \frac{\sqrt{n}}{\lVert\tilde\varphi\rVert^3} \frac{(0.5-\alpha)}{10.89+20.04\delta_0^{-1}+8.577\delta_0^{-2}}\]

We could use here the lower bound on $\lVert\tilde\varphi\rVert$ to replace the left condition by a condition of the same form as the right one, but this would be too strong when $\lVert\tilde\varphi\rVert$ is far from the bound. We will make a choice which will be better when $\lVert\tilde\varphi\rVert$ is of the order of $1$, by replacing the above conditions by the more stringent 
\[ T \le \frac{\sqrt{n}}{\max\{\lVert\tilde\varphi\rVert,\lVert\tilde\varphi\rVert^3\}} \min\Big\{\log\Big(1+\frac{\delta_0^2}{13+12\delta_0}\Big) , \frac{(0.5-\alpha)}{10.89+20.04\delta_0^{-1}+8.577\delta_0^{-2}} \Big\}.\]
In the $\min$, the first term is larger than $0.98\delta_0^2/(13+12\delta_0)$ which is easily seen to be larger than the second term for all $\delta_0$. We thus take
\[T = \frac{\sqrt{n}(0.5-\alpha)}{\max\{\lVert\tilde\varphi\rVert,\lVert\tilde\varphi\rVert^3\} \big(10.89+20.04\delta_0^{-1}+8.577\delta_0^{-2}\big)}\]
and we obtain
\begin{align*}
\lVert F_n-G\rVert_\infty &\le \frac{2.64}{\pi\sqrt{n}}\int_0^{+\infty} e^{-0.9999\alpha t^2} \big(f\lVert \tilde\varphi\rVert^3 t^2 + g \lVert \tilde\varphi\rVert  + ht^3\big) \dd t \\ &\qquad +\frac{(33.193+61.082\delta_0^{-1}+26.082\delta_0^{-2})\max\{\lVert\tilde\varphi\rVert,\lVert\tilde\varphi\rVert^3\}}{(0.5-\alpha)\sqrt{n}}.
\end{align*}
Setting $\alpha'=0.9999\alpha$, we have for each $d=0,2,3$: 
\[\int_0^{+\infty} e^{-\alpha' t^2} t^d \dd t = {\alpha'}^{-\frac{d+1}{2}} \int_0^{+\infty} e^{-t^2} t^d \dd t=\frac12 {\alpha'}^{-\frac{d+1}{2}} \Gamma\Big(\frac{d+1}{2}\Big)\]
and thus:
\begin{align*}
\lVert F_n-G\rVert_\infty &\le \frac{1.32}{\pi\sqrt{n}}\big( f{\alpha'}^{-\frac32}\frac{\sqrt{\pi}}{2} \lVert \tilde\varphi\rVert^3+g{\alpha'}^{-\frac12}\sqrt{\pi} \lVert \tilde\varphi\rVert+h{\alpha'}^{-2} \big)\\
&\qquad + \frac{(33.193+61.082\delta_0^{-1}+26.082\delta_0^{-2})\max\{\lVert\tilde\varphi\rVert,\lVert\tilde\varphi\rVert^3\}}{(0.5-\alpha)\sqrt{n}}.
\end{align*}

We will now choose $\alpha$, by comparing the two most troublesome coefficients in the small $\delta_0$ regime; these coefficients are $\frac{0.66 f}{\sqrt{\pi}(0.9999\alpha)^{3/2}}$, which is close to $3.162 \delta_0^{-2}\alpha^{-3/2}$ (making us want to take $\alpha$ large), and $\frac{(33.193+61.082\delta_0^{-1}+26.082\delta_0^{-2})}{0.5-\alpha}$ which is close to $26.082\delta_0^{-2}/(0.5-\alpha)$ (and makes us want to take $\alpha$ small). Optimizing the sum of these coefficients leads us to take $\alpha=0.195$. We then get
\begin{align*}
\lVert F_n-G\rVert_\infty &\le \frac{1}{\sqrt{n}}\Big( (32.05+76.77\delta_0^{-1}+36.72\delta_0^{-2} )\lVert \tilde\varphi\rVert^3+(6.81+7.32\delta_0^{-1}) \lVert \tilde\varphi\rVert\\
&\qquad +0.02 + (108.83+200.27\delta_0^{-1}+85.52\delta_0^{-2})\max\{\lVert\tilde\varphi\rVert,\lVert\tilde\varphi\rVert^3\}\Big) \\
  &\le \frac{1}{\sqrt{n}}\big(0.02 + (148+284.36\delta_0^{-1}+123 \delta_0^{-2}) \max\{\lVert\tilde\varphi\rVert,\lVert\tilde\varphi\rVert^3\} \big)
\end{align*}
which yields Theorem \ref{theo:main-BE} after using Lemma \ref{lemm:away0} to get $0.02 \le 0.03 \lVert\tilde\varphi\rVert\delta_0^{-1}$.

\bibliographystyle{amsalpha}
\bibliography{concentration}

\newcommand{\etalchar}[1]{$^{#1}$}
\providecommand{\bysame}{\leavevmode\hbox to3em{\hrulefill}\thinspace}
\providecommand{\MR}{\relax\ifhmode\unskip\space\fi MR }
\providecommand{\MRhref}[2]{%
  \href{http://www.ams.org/mathscinet-getitem?mr=#1}{#2}
}
\providecommand{\href}[2]{#2}
\begin{thebibliography}{GKLMF15}

\bibitem[Bal00]{Baladi}
Viviane Baladi, \emph{Positive transfer operators and decay of correlations},
  Advanced Series in Nonlinear Dynamics, vol.~16, World Scientific Publishing
  Co., Inc., River Edge, NJ, 2000. \MR{1793194 (2001k:37035)}

\bibitem[Bol82]{Bolthausen}
Erwin Bolthausen, \emph{The {B}erry-{E}sseen theorem for strongly mixing
  {H}arris recurrent {M}arkov chains}, Probability Theory and Related Fields
  \textbf{60} (1982), no.~3, 283--289.

\bibitem[BT08]{Bruin-Todd}
Henk Bruin and Mike Todd, \emph{Equilibrium states for interval maps:
  potentials with {$\sup\phi-\inf\phi<h_{\text{top}}(f)$}}, Comm. Math. Phys.
  \textbf{283} (2008), no.~3, 579--611. \MR{2434739}

\bibitem[CG12]{chazottes2012optimal}
Jean-Ren\'e Chazottes and S\'ebastien Gou\"ezel, \emph{Optimal concentration
  inequalities for dynamical systems}, Comm. Math. Phys. \textbf{316} (2012),
  no.~3, 843--889. \MR{2993935}

\bibitem[CP90]{Coelho-Parry}
Zaqueu Coelho and William Parry, \emph{Central limit asymptotics for shifts of
  finite type}, Israel J. Math. \textbf{69} (1990), no.~2, 235--249.
  \MR{1045376}

\bibitem[CS09]{Cyr-Sarig09}
Van Cyr and Omri Sarig, \emph{Spectral gap and transience for {R}uelle
  operators on countable {M}arkov shifts}, Comm. Math. Phys. \textbf{292}
  (2009), no.~3, 637--666. \MR{2551790}

\bibitem[CV13]{CV}
A.~Castro and P.~Varandas, \emph{Equilibrium states for non-uniformly expanding
  maps: decay of correlations and strong stability}, Ann. Inst. H. Poincar\'e
  Anal. Non Lin\'eaire \textbf{30} (2013), no.~2, 225--249. \MR{3035975}

\bibitem[DF15]{dedeker2015deviation}
J\'er\^ome Dedecker and Xiequan Fan, \emph{Deviation inequalities for
  separately {L}ipschitz functionals of iterated random functions}, Stochastic
  Process. Appl. \textbf{125} (2015), no.~1, 60--90. \MR{3274692}

\bibitem[DG15]{dedeker2015subgaussian}
J\'er\^ome Dedecker and S\'ebastien Gou\"ezel, \emph{Subgaussian concentration
  inequalities for geometrically ergodic {M}arkov chains}, Electron. Commun.
  Probab. \textbf{20} (2015), no. 64, 12. \MR{3407208}

\bibitem[Dub11]{Dubois}
Lo{\"\i}c Dubois, \emph{An explicit {B}erry-{E}ss{\'e}en bound for uniformly
  expanding maps on the interval}, Israel Journal of Mathematics \textbf{186}
  (2011), no.~1, 221--250.

\bibitem[Erd39]{Erdos1939}
Paul Erd{\"o}s, \emph{On a family of symmetric {B}ernoulli convolutions},
  American Journal of Mathematics \textbf{61} (1939), no.~4, 974--976.

\bibitem[Fel66]{Feller}
William Feller, \emph{An introduction to probability theory and its
  applications. {V}ol. {II}}, John Wiley \& Sons, Inc., New York-London-Sydney,
  1966. \MR{0210154}

\bibitem[GD12]{Gomez-Dartnell2012}
David~M G{\'o}mez and Pablo Dartnell, \emph{Simple monte carlo integration with
  respect to {B}ernoulli convolutions}, Applications of Mathematics \textbf{57}
  (2012), no.~6, 617--626.

\bibitem[GKLMF15]{GKLM}
Paolo Giulietti, Beno\^{\i}t~R. Kloeckner, Artur~O. Lopes, and Diego
  Marcon~Farias, \emph{The calculus of thermodynamical formalism},
  arXiv:1508.01297, to appear in \textit{J. Eur. Math. Soc.}, 2015.

\bibitem[GO02]{Glynn2002}
Peter~W Glynn and Dirk Ormoneit, \emph{{H}oeffding's inequality for uniformly
  ergodic {M}arkov chains}, Statistics \& probability letters \textbf{56}
  (2002), no.~2, 143--146.

\bibitem[Gou15]{gouezel2015limit}
S\'ebastien Gou\"ezel, \emph{Limit theorems in dynamical systems using the
  spectral method}, Hyperbolic dynamics, fluctuations and large deviations,
  Proc. Sympos. Pure Math., vol.~89, Amer. Math. Soc., Providence, RI, 2015,
  pp.~161--193. \MR{3309098}

\bibitem[HH01]{HH}
Hubert Hennion and Lo{\"{\i}}c Herv{\'e}, \emph{Limit theorems for {M}arkov
  chains and stochastic properties of dynamical systems by quasi-compactness},
  Lecture Notes in Mathematics, vol. 1766, Springer-Verlag, Berlin, 2001.

\bibitem[JO10]{JO}
Ald\'eric Joulin and Yann Ollivier, \emph{Curvature, concentration and error
  estimates for {M}arkov chain {M}onte {C}arlo}, Ann. Probab. \textbf{38}
  (2010), no.~6, 2418--2442. \MR{2683634}

\bibitem[KL99]{Keller-Liverani}
Gerhard Keller and Carlangelo Liverani, \emph{Stability of the spectrum for
  transfer operators}, Annali della Scuola Normale Superiore di Pisa-Classe di
  Scienze \textbf{28} (1999), no.~1, 141--152.

\bibitem[KLMM05]{Kontoyiannis-LMM}
Ioannis Kontoyiannis, Luis~A Lastras-Montano, and Sean~P Meyn, \emph{Relative
  entropy and exponential deviation bounds for general {M}arkov chains},
  International Symposium on Information Theory, 2005, IEEE, 2005,
  pp.~1563--1567.

\bibitem[Klo17a]{K:HT}
Beno\^{\i}t~R. Kloeckner, \emph{Effective high-temperature estimates for
  intermittent maps}, To appear in Ergodic Theory Dynam. Systems,
  arXiv:1704.00586, 2017.

\bibitem[Klo17b]{K:perturbation}
\bysame, \emph{Effective perturbation theory for linear operators},
  arXiv:1703.09425, 2017.

\bibitem[Klo17c]{K:WeaklyExp}
\bysame, \emph{An optimal transportation approach to the decay of correlations
  for non-uniformly expanding maps}, arXiv:1711.08052, 2017.

\bibitem[Klo18]{K:examples}
\bysame, \emph{Toy examples for effective concentration bounds}, 2018.

\bibitem[KM12]{Kontoyiannis-M}
Ioannis Kontoyiannis and Sean~P Meyn, \emph{Geometric ergodicity and the
  spectral gap of non-reversible {M}arkov chains}, Probability Theory and
  Related Fields (2012), 1--13.

\bibitem[Lez98]{Lezaud}
Pascal Lezaud, \emph{Chernoff-type bound for finite {M}arkov chains}, Ann.
  Appl. Probab. \textbf{8} (1998), no.~3, 849--867. \MR{1627795}

\bibitem[Lez01]{Lezaud2001}
\bysame, \emph{{C}hernoff and {B}erry--{E}ss{\'e}en inequalities for {M}arkov
  processes}, ESAIM: Probability and Statistics \textbf{5} (2001), 183--201.

\bibitem[Liv01]{liverani2001rigorous}
Carlangelo Liverani, \emph{Rigorous numerical investigation of the statistical
  properties of piecewise expanding maps. {A} feasibility study}, Nonlinearity
  \textbf{14} (2001), no.~3, 463--490. \MR{1830903}

\bibitem[Nag57]{Nagaev1}
S.~V. Nagaev, \emph{Some limit theorems for stationary {M}arkov chains}, Teor.
  Veroyatnost. i Primenen. \textbf{2} (1957), 389--416. \MR{0094846}

\bibitem[Nag61]{Nagaev2}
\bysame, \emph{More exact limit theorems for homogeneous {M}arkov chains},
  Teor. Verojatnost. i Primenen. \textbf{6} (1961), 67--86. \MR{0131291}

\bibitem[Pau15]{Paulin2015}
Daniel Paulin, \emph{Concentration inequalities for {M}arkov chains by {M}arton
  couplings and spectral methods}, Electronic Journal of Probability
  \textbf{20} (2015).

\bibitem[Pau16]{Paulin2016}
\bysame, \emph{Mixing and concentration by {R}icci curvature}, Journal of
  Functional Analysis \textbf{270} (2016), no.~5, 1623--1662.

\bibitem[PSS00]{Peres2000}
Yuval Peres, Wilhelm Schlag, and Boris Solomyak, \emph{Sixty years of
  {B}ernoulli convolutions}, Progress in probability (2000), 39--68.

\bibitem[RR{\etalchar{+}}04]{Roberts-Rosenthal}
Gareth~O Roberts, Jeffrey~S Rosenthal, et~al., \emph{General state space
  {M}arkov chains and mcmc algorithms}, Probability Surveys \textbf{1} (2004),
  20--71.

\bibitem[Rue04]{Ruelle}
David Ruelle, \emph{Thermodynamic formalism}, second ed., Cambridge
  Mathematical Library, Cambridge University Press, Cambridge, 2004, The
  mathematical structures of equilibrium statistical mechanics. \MR{2129258
  (2006a:82008)}

\bibitem[Sol95]{Solomyak1995}
Boris Solomyak, \emph{On the random series $\sum\pm\lambda^n$ (an {E}rd\"os
  problem)}, Annals of Mathematics (1995), 611--625.

\bibitem[Tyu11]{Tyurin}
I.~S. Tyurin, \emph{Improvement of the remainder in the {L}yapunov theorem},
  Teor. Veroyatn. Primen. \textbf{56} (2011), no.~4, 808--811. \MR{3137072}

\bibitem[WH17]{watanabe2017finite}
Shun Watanabe and Masahito Hayashi, \emph{Finite-length analysis on tail
  probability for {M}arkov chain and application to simple hypothesis testing},
  Ann. Appl. Probab. \textbf{27} (2017), no.~2, 811--845. \MR{3655854}

\end{thebibliography}
\end{document}